\pdfoutput=1
\RequirePackage{ifpdf}
\ifpdf 
\documentclass[pdftex]{sigma}
\else
\documentclass{sigma}
\fi

\numberwithin{equation}{section}

\newtheorem{Theorem}{Theorem}[section]
\newtheorem*{Theorem*}{Theorem}
\newtheorem{Corollary}[Theorem]{Corollary}
\newtheorem{Lemma}[Theorem]{Lemma}
\newtheorem{Proposition}[Theorem]{Proposition}
 { \theoremstyle{definition}
\newtheorem{Definition}[Theorem]{Definition}

\newtheorem{Remark}[Theorem]{Remark} }

\newtheorem{Th}{Theorem}[section]
\newtheorem{Cor}{Corollary}[section]
\newtheorem{Prop}{Proposition}[section]
\newtheorem{Lem}{Lemma}[section]
\theoremstyle{definition}
\newtheorem{Def}{Definition}[section]
\newtheorem{Rem}{Remark}[section]
\newtheorem{Ex}{Example}[section]

\newcommand{\bet}{\begin{Th}}
\newcommand{\ent}{\stepcounter{Cor}
 \stepcounter{Prop}\stepcounter{Lem}\stepcounter{Def}
 \stepcounter{Rem}\stepcounter{Ex}\end{Th}}

\newcommand{\bec}{\begin{Cor}}
\newcommand{\enc}{\stepcounter{Th}
 \stepcounter{Prop}\stepcounter{Lem}\stepcounter{Def}
 \stepcounter{Rem}\stepcounter{Ex}\end{Cor}}
\newcommand{\bep}{\begin{Prop}}
\newcommand{\enp}{\stepcounter{Th}
 \stepcounter{Cor}\stepcounter{Lem}\stepcounter{Def}
 \stepcounter{Rem}\stepcounter{Ex}\end{Prop}}
\newcommand{\bel}{\begin{Lem}}
\newcommand{\enl}{\stepcounter{Th}
 \stepcounter{Cor}\stepcounter{Prop}\stepcounter{Def}
 \stepcounter{Rem}\stepcounter{Ex}\end{Lem}}
\newcommand{\bef}{\begin{Def}}
\newcommand{\enf}{\stepcounter{Th}
 \stepcounter{Cor}\stepcounter{Prop}\stepcounter{Lem}
 \stepcounter{Rem}\stepcounter{Ex}\end{Def}}
\newcommand{\ber}{\begin{Rem}}
\newcommand{\enr}{
 \stepcounter{Th}\stepcounter{Cor}\stepcounter{Prop}
 \stepcounter{Lem}\stepcounter{Def}\stepcounter{Ex}\end{Rem}}
\newcommand{\bee}{\begin{Ex}}
\newcommand{\ene}{
 \stepcounter{Th}\stepcounter{Cor}\stepcounter{Prop}
 \stepcounter{Lem}\stepcounter{Def}\stepcounter{Rem}\end{Ex}}

\newcommand{\R}{\mathbb{R}}
\newcommand{\C}{\mathbb{C}}

\newcommand{\KK}{\mathbb{K}}

\newcommand{\PP}{\mathbb{P}}

\newcommand{\uuu}{\boldsymbol{u}}

\newcommand{\0}{\boldsymbol{0}}

\newcommand{\pa}{\partial}

\newcommand{\OO}{{\mathbb O}}

\newcommand{\rank}{\operatorname{rank}}

\newcommand{\Spin}{{\mbox {\rm Spin}}}

\newcommand{\SVC}{{\mbox {\rm SVC}}}

\begin{document}

\allowdisplaybreaks

\newcommand{\arXivNumber}{2501.02789}

\renewcommand{\PaperNumber}{076}

\FirstPageHeading

\ShortArticleName{Prolongation of $(8,15)$-Distribution of Type $F_4$ by Singular Curves}

\ArticleName{Prolongation of $\boldsymbol{(8,15)}$-Distribution of Type $\boldsymbol{F_4}$\\ by Singular Curves}

\Author{Goo ISHIKAWA~$^{\rm a}$ and Yoshinori MACHIDA~$^{\rm b}$}

\AuthorNameForHeading{G.~Ishikawa and Y.~Machida}

\Address{$^{\rm a)}$~Department of Mathematics, Hokkaido University, Kita 10 Nishi 8, Kita-ku,\\
\hphantom{$^{\rm a)}$}~Sapporo 060-0810, Japan}
\EmailD{\href{mailto:ishikawa@math.sci.hokudai.ac.jp}{ishikawa@math.sci.hokudai.ac.jp}}

\Address{$^{\rm b)}$~Department of Mathematics, Faculty of Science, Shizuoka University,\\
\hphantom{$^{\rm b)}$}~836, Ohya, Suruga-ku, Shizuoka 422-8529, Japan}
\EmailD{\href{mailto:machida.yoshinori@shizuoka.ac.jp}{machida.yoshinori@shizuoka.ac.jp}, \href{mailto:yomachi212@gmail.com}{yomachi212@gmail.com}}

\ArticleDates{Received January 30, 2025, in final form September 12, 2025; Published online September 18, 2025}

\Abstract{Cartan gives the model of $(8, 15)$-distribution with the exceptional simple Lie algebra $F_4$ as its symmetry algebra in his paper (1893), which is published one year before his thesis. In the present paper, we study abnormal extremals (singular curves) of Cartan's model from viewpoints of sub-Riemannian geometry and geometric control theory.
Then we construct the prolongation of Cartan's model based on the data related to its singular curves, and obtain the nilpotent graded Lie algebra which is isomorphic to the negative part of the graded Lie algebra $F_4$.}

\Keywords{exceptional Lie algebra; singular curve; constrained Hamiltonian equation}

\Classification{53C17; 58A30; 17B25; 34H05; 37J37; 49K15; 53D25}

\section{Introduction}

Let $M$ be a manifold of dimension $15$ and $D \subset TM$ a distribution, i.e., a vector subbundle of
the tangent bundle $TM$ of rank $8$.
Then $D$ is called an $(8, 15)$-distribution
if ${\mathcal D} + [{\mathcal D}, {\mathcal D}] = {\mathcal{TM}}$ for the sheave ${\mathcal D}$ (resp.
${\mathcal{TM}}$) of local sections to $D$ (resp.\ $TM$).
In this paper, we study a special class of $(8, 15)$-distributions related to the simple Lie group $F_4$.

Distributions are important subjects in manifold theory and global analysis. They are studied
also related to the theory of Lie groups, Lie algebras and their representations.
Then the theory of prolongations and equivalence problems of distributions are established by many authors (see, for instance,
\cite{BCGGG,Montgomery, Morimoto, Tanaka1}).
On symmetries for distributions, there are well-known several
powerful and beautiful methods to investigate, based on differential geometry and representation theory;
Cartan's prolongation, Tanaka's prolongation, and Kostant's theorem on Bott--Borel--Weil theory
and so on \cite{Cartan1,Helgason, Hwang, LM, Kostant, Tanaka1, Tanaka2, Yamaguchi}.

We provide, in this paper, a way of prolongations of $(8, 15)$-distributions of type $F_4$
via the notion of {\it abnormal extremals} or {\it singular curves} and related objects from viewpoints of
sub-Riemannian geometry and geometric control theory \cite{AS, LS, Montgomery0, Montgomery}
which recovers several results explicitly. The relations of our constructions with those by the method of representation theory are presented in Remark~\ref{Remark-representation-theory} of Section~\ref{Singular-velocity} in our paper.

The singular curves or abnormal extremals are extensively used to study distributions by many authors (see, for instance,
\cite{AZ, BH, DZ1, DZ2}).
In the previous papers (see \cite{IKTY, IKY1,IMT}), we study $(2, 3, 5)$-distributions or
Cartan distributions \cite{Bryant, Cartan1} using singular curves.
Here a~$(2, 3, 5)$-distribution means a distribution $D$ of rank $2$ on a $5$-dimensional manifold $M$
such that \smash{${{\mathcal D}^{(2)} := {\mathcal D} + [{\mathcal D}, {\mathcal D}]}$}
becomes the sheaf of local sections of a distribution \smash{$D^{(2)}$} of rank $3$ and that
\smash{${{\mathcal{TM}} = {\mathcal D}^{(3)} := {\mathcal D}^{(2)} + \bigl[{\mathcal D}, {\mathcal D}^{(2)}\bigr]}$}.
Then we show the prolongation using the cone of singular curves of any $(2, 3, 5)$-distribution has
the nilpotent gradation algebra which is isomorphic to the negative part of the graded simple Lie algebra $G_2$.
Note that the prolongation procedure is a partial case of twistor construction in the general framework of
parabolic geometry~\cite{BE, CS}.\looseness=1

In his book \cite{Montgomery} on sub-Riemannian geometry,
Montgomery gives expositions on $(4, 7)$-distributions.
In particular, Montgomery classifies $(4, 7)$-distributions into elliptic, hyperbolic and parabolic $(4, 7)$-distributions and
shows the non-existence of non-trivial singular curves for elliptic $(4, 7)$-distribution. Moreover, he develops
Cartan's approach for $(4, 7)$-distributions and studies their symmetry groups.
In the previous paper \cite{IM}, we study hyperbolic $(4, 7)$-distributions and their prolongations via
the cone of singular curves.
Then we observe, contrary to the case of $(2, 3, 5)$-distributions,
the isomorphism classes of the nilpotent graded Lie algebra of prolongations are never unique and
then we specifies the class of $C_3$-$(4, 7)$-distributions by the condition that
 the graded algebra associated to the $(4, 7)$-distribution
 after prolongation is isomorphic to the negative part of the simple Lie algebra $C_3$.

Cartan, in his paper \cite{Cartan1} which is published one year before his thesis \cite{Cartan2},
gives the model of $(8, 15)$-distribution
whose infinitesimal symmetry algebra is the simple Lie algebra $F_4$.
The purpose of the present paper is to study Cartan's model of $(8, 15)$-distribution from viewpoints of
sub-Riemannian geometry and geometric control theory.
We construct its prolongation using the data related to abnormal or singular curves,
and verify that the prolonged nilpotent graded algebra obtained
by our method is isomorphic to the negative part of the simple Lie algebra~$F_4$.\looseness=1

Note that the complex simple Lie algebra $F_4$ has three real forms;
one {\it compact} type and two {\it non-compact} types
denoted as \smash{$F_{4(4)}$} and as \smash{$F_{4(-20)}$} (see \cite{Cartan1, Cartan2, Cartan3, Gantmacher, Jacobson}).
In \cite{Helgason}, \smash{$F_{4(4)}$} (resp.~\smash{$F_{4(-20)}$}) is denoted by $F_4I$ (resp.\ $F_4II$),
and in \cite{Harvey}, by $\widetilde{F}_4$ (resp.\ \smash{$F_4'$}).
Cartan's model, which we treat in the present paper, gives the
$(8, 15)$-distribution corresponding to $F_{4(4)}$, which maybe called
the ``hyperbolic'' $F_4$-$(8, 15)$-distribution.
Nurowski \cite{Nurowski}
has given the explicit models of $(8, 15)$-distributions of type $F_4$ and $(16, 24)$-distributions of type
$E_6$. Though we do not touch the details here,
it can be observed that the real $(8, 15)$-distribution of type~\smash{$F_{4(-20)}$} in Nurowski's normal form has
the canonical definite conformal metric and it has no nontrivial singular curves
 (cf.\ Sections~\ref{Abnormal-bi-extremals-general} and~\ref{Singular-velocity} of this paper).
Thus Nurowski's $(8, 15)$-distribution of type~$F_{4(-20)}$ can be called
 ``elliptic'' $F_4$-$(8, 15)$-distribution. Refer \cite{Nurowski} also for related references and historical remarks.
 Note also that both $(8, 15)$-distributions of type
 \smash{$F_{4(4)}$} and \smash{$F_{4(-20)}$} appear, as two cases of real simple Lie algebras,
 in the classification of certain sub-Riemannian structures in \cite{AS2, MKMV}.

In Section~\ref{Cartan-model}, we recall Cartan's model $\bigl(\KK^{15}, D\bigr)$ of $(8, 15)$-distribution
associated to the simple Lie algebra $F_4$.
The basics on sub-Riemannian geometry and geometric control theory which
we need in this paper are given in Section~\ref{Abnormal-bi-extremals-general}.
We study the singular curves of Cartan's model and show that there exist canonically
the conformal metrics on ${D \subset T\KK^{15}}$ and on ${D^\perp \subset T^*\KK^{15}}$ in~Section~\ref{Singular-velocity}.
In Section~\ref{Null-flags}, we construct null-flag manifold of dimension $9$ which prolongs $\bigl(\KK^{15}, D\bigr)$
to $(W, E)$ so that $\dim(W) = 24$ and $E$ is of rank $4$.
In Section~\ref{Prolongation-of-Cartan's-$(8, 15)$-distribution}, we show that $E$
has the small growth vector $(4,7,10,13,16,18,20,21,22,23,24)$ and the gradation algebra of $E$
is isomorphic to the negative part of the simple graded Lie algebra
with respect to filtration defined by the set of all roots of $F_4$.
In Section~\ref{$(8, 15)$-distributions of type $F_4$}, we
introduce the class of $(8, 15)$-distributions of type~$F_4$ regarding the arguments of previous sections
and show that also the gradation algebra of the prolongation of any $(8, 15)$-distributions of type $F_4$
is isomorphic to the negative part of the simple Lie algebra~$F_4$ with respect
to the filtration defined by the set of all roots of $F_4$ (Theorem~\ref{Main-Theorem}).\looseness=1

In this paper, all manifolds and maps are supposed to be of class $C^\infty$ unless otherwise stated.\looseness=1

\section[Cartan's model of (8, 15)-distributions of type F\_4]{Cartan's model of $\boldsymbol{(8, 15)}$-distributions of type $\boldsymbol{F_4}$}
\label{Cartan-model}

We recall Cartan's model of $(8, 15)$-distribution \cite{Cartan1, Yamaguchi}
which has, as the infinitesimal symmetries, the simple Lie algebra $F_4$:
\begin{center}
\includegraphics[width=4truecm,clip,
bb=7 10 360 40
]{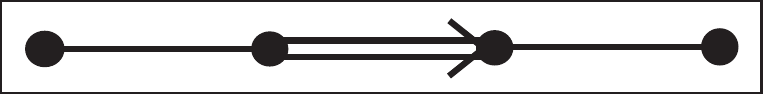}
\\
{\rm The Dynkin diagram of $F_4$.}
\end{center}

As for the exceptional Lie algebra $F_4$, see, for instance, also \cite{Adams2, Adams1, Bourbaki, Humphreys,Sato}.

The model of $(8, 15)$-distributions found by Cartan
is derived from the homogeneous space by the parabolic subgroup of the simple Lie group $F_4$
which corresponds to $\{ \alpha_4\}$ for the simple roots $\alpha_1$, $\alpha_2$, $\alpha_3$, $\alpha_4$
\cite{Cartan1, Yamaguchi}:
\begin{center}
\includegraphics[width=4truecm, clip,
bb=15 10 365 70
]{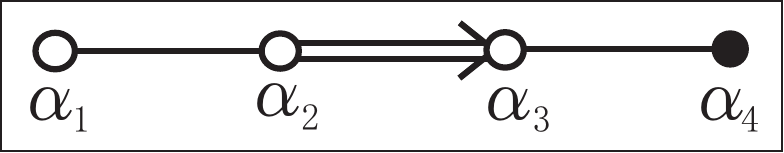}
\end{center}
Here we have simply marked the corresponding root in black to the parabolic subgroup,
which not meant, say, the Satake diagram.
Note that, as the standard way, a cross under the node can be used to indicate a parabolic subgroup as in~\cite{BE}.

Let $\KK = \R$ or $\C$.
On $\KK^{15}$ with the system of coordinates
$z$, $x_1$, $x_2$, $x_3$, $x_4$, $y_1$, $y_2$, $y_3$, $y_4$, $x_{ij}$, $1 \leq i < j \leq 4$, and consider the $C^\infty$ (resp.\ holomorphic) $1$-forms
\begin{gather*}
\omega = {\rm d}z - y_1{\rm d}x_1 - y_2{\rm d}x_2 - y_3{\rm d}x_3 - y_4{\rm d}x_4,
\\
\omega_{ij} =
{\rm d}x_{ij} - (x_i{\rm d}x_j - x_j{\rm d}x_i + y_h{\rm d}y_k - y_k{\rm d}y_h), \qquad 1 \leq i < j \leq 4,
\end{gather*}
where $(i, j, h, k)$ is an even permutation of $(1, 2, 3, 4)$.
Let
\begin{gather*}
Z,\ X_{12},\ X_{13},\ X_{14},\ X_{23},\ X_{24},\ X_{34},\ X_1,\ X_2,\ X_3,\ X_4,\ Y_1,\ Y_2,\ Y_3,\ Y_4
\end{gather*}
be the dual frame of $T\KK^{15}$ to the frame
\[
\omega,\ \omega_{12},\ \omega_{13},\ \omega_{14},\ \omega_{23},\ \omega_{24},\ \omega_{34},\
{\rm d}x_1,\ {\rm d}x_2,\ {\rm d}x_3,\ {\rm d}x_4,\ {\rm d}y_1,\ {\rm d}y_2,\ {\rm d}y_3,\ {\rm d}y_4
\]
of $T^*\KK^{15}$.
Then $D \subset T\KK^{15}$ is defined as the distribution generated by $X_1$, $X_2$, $X_3$, $X_4$, $Y_1$, $Y_2$, $Y_3$, $Y_4$.
Explicitly the distribution $D \subset T\KK^{15}$ has the system of generators
\begin{gather*}
X_1 = \frac{\pa}{\pa x_1} + y_1\frac{\pa}{\pa z} - x_2\frac{\pa}{\pa x_{12}} - x_3\frac{\pa}{\pa x_{13}} -
x_4\frac{\pa}{\pa x_{14}},
\\
X_2 = \frac{\pa}{\pa x_2} + y_2\frac{\pa}{\pa z} + x_1\frac{\pa}{\pa x_{12}} - x_3\frac{\pa}{\pa x_{23}} -
x_4\frac{\pa}{\pa x_{24}},
\\
X_3 = \frac{\pa}{\pa x_3} + y_3\frac{\pa}{\pa z} + x_1\frac{\pa}{\pa x_{13}} + x_2\frac{\pa}{\pa x_{23}} -
x_4\frac{\pa}{\pa x_{34}},
\\
X_4 = \frac{\pa}{\pa x_4} + y_4\frac{\pa}{\pa z} + x_1\frac{\pa}{\pa x_{14}} + x_2\frac{\pa}{\pa x_{24}} +
x_3\frac{\pa}{\pa x_{34}},
\\
Y_1 = \frac{\pa}{\pa y_1} - y_4\frac{\pa}{\pa x_{23}} + y_3\frac{\pa}{\pa x_{24}} - y_2\frac{\pa}{\pa x_{34}},
\\
Y_2 = \frac{\pa}{\pa y_2} + y_4\frac{\pa}{\pa x_{13}} - y_3\frac{\pa}{\pa x_{14}} + y_1\frac{\pa}{\pa x_{34}},
\\
Y_3 = \frac{\pa}{\pa y_3} - y_4\frac{\pa}{\pa x_{12}} + y_2\frac{\pa}{\pa x_{14}} - y_1\frac{\pa}{\pa x_{24}},
\\
Y_4 = \frac{\pa}{\pa y_4} + y_3\frac{\pa}{\pa x_{12}} - y_2\frac{\pa}{\pa x_{13}} + y_1\frac{\pa}{\pa x_{23}}.
\end{gather*}

Moreover, we have that
\smash{$Z = \frac{\pa}{\pa z}$} and \smash{$X_{ij} = \frac{\pa}{\pa x_{ij}}$}, $1 \leq i < j \leq 4$.
Then we get the following bracket relations:
\begin{gather*}
[X_1, X_2] = 2X_{12},\qquad [X_1, X_3] = 2X_{13},\qquad [X_1, X_4] = 2X_{14},
\\
 [X_2, X_3] = 2X_{23},\qquad [X_2, X_4] = 2X_{24},
\\
 [X_3, X_4] = 2X_{34},
\\
[Y_1, Y_2] = 2X_{34},\qquad [Y_1, Y_3] = -2X_{24},\qquad [Y_1, Y_4] = 2X_{23},
\\
 [Y_2, Y_3] = 2X_{14},\qquad [Y_2, Y_4] = -2X_{13},
\\
 [Y_3, Y_4] = 2X_{12},
\\
[Y_1, X_1] = [Y_2, X_2] = [Y_3, X_3] = [Y_4, X_4] = Z, \qquad [Y_i, X_j] = 0, \qquad i \not= j.
\end{gather*}
Moreover, we have
\[
[X_i, X_{jk}] = 0, \qquad [Y_i, X_{jk}] = 0, \qquad [X_i, Z] = 0, \qquad [Y_i, Z] = 0 \qquad \text{for any} \ i, \ j, \ k.
\]

\begin{Remark}
We set, for $1 \leq i < j \leq 4$, a sub-distribution
$D_{ij} = \langle X_i, X_j, Y_h, Y_k, X_{ij}\rangle$ of $D^{(2)}$, where $(i, j, h, k)$ is a permutation of $(1, 2, 3, 4)$.
Then we see each $D_{ij}$ is completely integrable
and each leaf of the foliation induced by $D_{ij}$ of $\KK^{15}$ has a contact structure.
Thus we have six contact foliations in $\KK^{15}$.
For example, for $i = 1$, $j = 2$, then the contact foliation is given by the Pfaff system
\begin{gather*}
{\rm d}z - y_1{\rm d}x_1 - y_2{\rm d}x_2 = 0,\qquad {\rm d}x_3 = 0,\qquad {\rm d}x_4 = 0,\qquad {\rm d}y_1 = 0,\qquad {\rm d}y_2 = 0,
\\
{\rm d}x_{13} + x_3{\rm d}x_1 + y_2{\rm d}y_4 = 0,\qquad {\rm d}x_{14} + x_4{\rm d}x_1 - y_2{\rm d}y_3 = 0,\qquad
{\rm d}x_{23} + x_3{\rm d}x_2 + y_1{\rm d}y_4 = 0,
\\
{\rm d}x_{24} + x_4{\rm d}x_2 + y_1{\rm d}y_3 = 0,\qquad {\rm d}x_{34} = 0,
\end{gather*}
and with the $1$-form
\[
{\rm d}x_{12} + x_2{\rm d}x_1 - x_1{\rm d}x_2 + y_4{\rm d}y_3 - y_3{\rm d}y_4,
\]
which gives a contact structure on each leaf of the foliation defined by $D_{12}$.
\end{Remark}

\section{Abnormal bi-extremals and singular curves of distributions}
\label{Abnormal-bi-extremals-general}

Here we recall several notions in geometric control theory and sub-Riemannian geometry.
For details, consult, for instance, the references \cite{AS, BC, Montgomery0, Montgomery}.

Let $M$ be a real $C^\infty$ manifold,
$D \subset TM$ a distribution endowed with a positive definite met\-ric~${g\colon D\otimes D \to \R}$
on a manifold $M$, and $\gamma\colon [a, b] \to M$ an absolutely continuous curve sat\-is\-fy\-ing~$\dot{\gamma}(t) \in D_{\gamma(t)}$
for almost all $t \in I$, which is called a {\it $D$-integral curve}.
Then the arc-length of $\gamma$ is defined by
\smash{$L(\gamma) := \int_a^b \sqrt{g(\dot{\gamma}(t), \dot{\gamma}(t))} {\rm d}t$}.
A curve $\gamma$ is called a {\it $D$-geodesic} if it minimises the arc-length locally.

Let $\rank(D) = r$ and, just for simplicity, $\xi_1, \dots, \xi_r$ be an orthonormal frame of $(D, g)$ over~$M$.
Then we define $F\colon D \cong M\times \R^r \to TM$ by \smash{$F(x, u) = \sum_{i=1}^r u_i\xi_i(x)$}.

Consider the optimal control problem for the energy function on $D$ defined by
\[
e = \frac{1}{2} g\Biggl(\sum_{i=1}^r u_i \xi_i(x), \ \sum_{i=1}^r u_i \xi_i(x)\Biggr) = \frac{1}{2} \sum_{i=1}^r u_i(t)^2.
\]
Note that the problem of minimising arc-length and that of minimising energy function are known to be equivalent
up to re-parametrisations \cite{AS, Montgomery}.
Then the Hamiltonian function on~$(D \times_M T^*M) \times \R$ is given by
\[
H(x, p, u, p^0) = \Biggl\langle p, \sum_{i=1}^r u_i\xi_i(x)\Biggr\rangle + p^0\Biggl(\frac{1}{2}\sum_{i=1}^r u_i^2\Biggr).
\]
Here $D \times_M T^*M = \{ (x, u), (x', p) \in D\times T^*M \mid x = x'\} \cong T^*M \times \R^r$ and $p^0$ is an additional parameter.

Regarding the optimal control problem for minimising the energy function of $D$-integrable curves,
we have, by Pontryagin's maximum principle,
if $\gamma$ is a $D$-geodesic, then, for $\dot{\gamma}(t) = (x(t), u(t))$,
there exists a Lipschitz curve $(x(t), p(t)) \in T^*M$ and non-positive constant $p^0 \leq 0$
such that the following constrained Hamiltonian equation in terms of $H = H\bigl(x, p, u, p^0\bigr)$
is satisfied:
\begin{gather*}
\dot{x}_i(t) = \frac{\pa H}{\pa p_i}\bigl(x(t), p(t), u(t), p^0\bigr), \qquad 1 \leq i \leq m,
\\
\dot{p}_i(t) = - \frac{\pa H}{\pa x_i}\bigl(x(t), p(t), u(t), p^0\bigr), \qquad 1 \leq i \leq m,
\end{gather*}
with constraints
\smash{$\frac{\pa H}{\pa u_j}\bigl(x(t), p(t), u(t), p^0\bigr) = 0$}, $1 \leq j \leq r$, \smash{$\bigl(p(t), p^0\bigr) \not= 0$}.

If \smash{$p^0 < 0$}, then
the curve $(x(t), p(t))$ (resp.\ $x(t)$)
of a solution of the above constrained Hamiltonian equation
is called a {\it normal bi-extremal}
(resp.\ {\it normal extremal}) respectively.
If~${p^0 = 0}$, then bi-extremals and extremals are called {\it abnormal}.
Note that the notion of abnormal \mbox{(bi-)extremals} is independent of the metric $g$ on $D$ and
depends only on the distribution $D$.

The constraint
\smash{$\frac{\pa H}{\pa u_j} = 0$} is equivalent to that
\smash{$p^0u_j = - \langle p, \xi_j(x)\rangle$}.
In the normal case, i.e., $p^0 < 0$, we have
\smash{$u_j = - \frac{1}{p^0}\langle p, \xi_j(x)\rangle$}. Because the Hamiltonian
is linear on $\bigl(p, p^0\bigr)$, by normalising as $p^0 = -1$, we have
\smash{$H = \frac{1}{2}\sum_{i=1}^r \langle p, \xi_i(x)\rangle^2$}.

For abnormal extremals, the constrained Hamiltonian equation reads as
\begin{gather*}
\dot{x} = \ u_1\xi_1(x) + u_2\xi_2(x) + \cdots + u_r\xi_i(x),
\\
\dot{p} = - \biggl(u_1\frac{\pa H_{\xi_1}}{\pa x} + u_2\frac{\pa H_{\xi_2}}{\pa x} + \cdots +
u_r\frac{\pa H_{\xi_r}}{\pa x}\biggr),
\end{gather*}
with constraints
$H_{\xi_1} = 0, H_{\xi_2} = 0, \dots, H_{\xi_r} = 0$ and $p \not= 0$,
where $H_{\xi_i}(x, p) := \langle p, \xi_i(x)\rangle$.

Given a distribution $D \subset TM$, for any $x \in M$, we define the subbundle $D^\perp \subset T^*M$ by
\[
D^\perp_x := \{ \alpha \in T^*_xM \mid \langle \alpha, v\rangle = 0, {\mbox{\rm\ for any\ }} v \in D_x\}.
\]
Then the above constraints mean that $p(t) \in D^\perp_{x(t)}$.

The notion of abnormal extremals coincides with that of singular curves, i.e.,
critical points of the end-point mapping \cite{Montgomery0, Montgomery}.
Let $x_0 \in M$ and $I = [a, b]$ an interval. Let
$\Omega$ be the set of Lipschitz continuous curves $\gamma\colon I \to M$ with
$\dot{\gamma}(t) \in D_{\gamma(t)}$ for almost all $t \in I$, which is called a $D$-integral curve,
and $\gamma(a) = x_0$.
Then the {\it endpoint mapping}
${\rm End}\colon \Omega \to M$
is defined by ${\rm End}(\gamma) := \gamma(b)$.
A curve $\gamma \in \Omega$ is called a {\it $D$-singular curve} if $\gamma$ is
a critical point of the endpoint mapping,
i.e., the differential map
${\rm d}_{\gamma}{\rm End}\colon T_{\gamma}\Omega \to T_{\gamma(b)}M$
is not surjective, for an~appropriate manifold structure of $\Omega$ (and $M$).

We introduce the key notion of the present paper.

\begin{Definition}
We define the {\it singular velocity cone} $\SVC(D) \subset TM$ of a distribution ${D \subset TM}$ by
the set of tangent vectors $v \in T_xM$, $x \in M$ such that there exists a $D$-singular curve
$\gamma\colon (\R, 0) \to (M, x)$ with $\gamma'(0) = v$.
\end{Definition}

Note that $\SVC(D)$ is a cone field over $M$, i.e., $\SVC(D)$ is invariant under
the fibrewise $\R^\times$-multiplication on $TM$.

The following lemma is used in the following sections. We have given a proof using coordinates to make sure.

\begin{Lemma}[{\cite{AS}} and {\cite[Section~4.2]{BC}}]
\label{Singular-curve-lemma} For a distribution $D$ generated by $\xi_1, \dots, \xi_r$,
we have, along abnormal bi-extremals $(x(t), p(t))$ and corresponding $u(t)$, that
\[
\frac{{\rm d}}{{\rm d}t}H_{\xi_i}(t) = \sum_{j=1}^r u_j(t) H_{[\xi_i, \xi_j]}(t), \qquad 1 \leq i \leq r.
\]
\end{Lemma}
\begin{proof}
We put \smash{$p = \sum_{j=1}^r p_j {\rm d}x_j$} and \smash{$\xi_i = \sum_{k=1}^r \xi_{ik}\frac{\pa}{\pa x_k}$}.
Then \smash{$H(x, p, u) = \sum_{1 \leq i, j \leq r} u_ip_j\xi_{ij}(x)$} and \smash{$H_{\xi_i} = \sum_{j=1}^r p_j\xi_{ij}(x)$}.
By the Hamiltonian equation, for $1 \leq i \leq r$, we have
\begin{align*}
\frac{{\rm d}}{{\rm d}t}H_{\xi_i}(t) & = \sum_{j=1}^r \bigl(p_j'\xi_{ij} + p_j\xi_{ij}'\bigr) =
\sum_{j=1}^r \Biggl(p_j'\xi_{ij} + \sum_{\ell =1}^r p_j\frac{\pa \xi_{ij}}{\pa x_\ell}x_{\ell}'\Biggr)
\\
& = \sum_{j=1}^r \Biggl(- \frac{\pa H}{\pa x_j}\xi_{ij}
+ \sum_{\ell =1}^r p_j\frac{\pa \xi_{ij}}{\pa x_\ell}\frac{\pa H}{\pa p_\ell}\Biggr)
 = - \sum_{k\ell j} u_kp_\ell \frac{\pa \xi_{k\ell}}{\pa x_j}\xi_{ij}
+ \sum_{j\ell k} p_j\frac{\pa \xi_{ij}}{\pa x_\ell}u_k\xi_{k\ell}
\\
& =
 - \sum_{k\ell j} u_kp_\ell \frac{\pa \xi_{k\ell}}{\pa x_j}\xi_{ij}
+ \sum_{\ell j k} p_\ell\frac{\pa \xi_{i\ell}}{\pa x_j}u_k\xi_{kj}
=
\sum_{k\ell} u_kp_\ell\Biggl(
\sum_{j=1}^r\biggl(
\xi_{ij}\frac{\pa \xi_{k\ell}}{\pa x_j} - \xi_{kj}\frac{\pa \xi_{i\ell}}{\pa x_j}
\biggr)
\Biggr)
\\
& = \sum_{k=1}^r u_k\langle p, [\xi_i, \xi_k]\rangle = \sum_{j=1}^r u_jH_{[\xi_i, \xi_j]}.
\tag*{\qed}
\end{align*}
\renewcommand{\qed}{}
\end{proof}

\begin{Remark}
We have defined the notion of abnormal (bi-)extremals and singular curves over the real.
In the complex analytic case $\KK = \C$, we can (and do) define abnormal (bi-)extremals and singular curves, forgetting
about end-point mapping,
just by the complex analytic constrained Hamiltonian equation for a complex analytic distribution $D \subset TM$
over a complex analytic manifold $M$, which is defined similarly as explained in this section.
\end{Remark}

\section[Conformal metric on Cartan's (8, 15)-distribution and singular velocity cone]{Conformal metric on Cartan's $\boldsymbol{(8, 15)}$-distribution\\ and singular velocity cone}
\label{Singular-velocity}

Let us determine the singular curves of Cartan's model $\bigl(\KK^{15}, D\bigr)$
explained in Section~\ref{Cartan-model}.
On the cotangent bundle $T^*\KK^{15}$ with base coordinates
$z$, $x_1$, $x_2$, $x_3$, $x_4$, $y_1$, $y_2$, $y_3$, $y_4$, $x_{ij}$, ${1 \leq i < j \leq 4}$ and
fiber coordinates
$s$, $p_1$, $p_2$, $p_3$, $p_4$, $q_1$, $q_2$, $q_3$, $q_4$, $r_{ij}$, $1 \leq i < j \leq 4$,
we have the Hamiltonian of the distribution $D \subset TX$,
\[
H = u_1H_{X_1} + u_2H_{X_2} + u_3H_{X_3} + u_4H_{X_4} + v_1H_{Y_1} + v_2H_{Y_2} + v_3H_{Y_3} + v_4H_{Y_4},
\]
where
\begin{alignat*}{3}
&H_{X_1} = p_1 + y_1s - x_2r_{12} - x_3r_{13} - x_4r_{14}, \qquad&&
H_{X_2} = p_2 + y_2s + x_1r_{12} - x_3r_{23} - x_4r_{24},&
\\
&H_{X_3} = p_3 + y_3s + x_1r_{13} + x_2r_{23} - x_4r_{34}, \qquad&&
H_{X_4} = p_4 + y_4s + x_1r_{14} + x_2r_{24} + x_3r_{34},&
\\
&H_{Y_1} = q_1 - y_4r_{23} + y_3r_{24} - y_2r_{34}, \qquad&&
H_{Y_2} = q_2 + y_4r_{13} - y_3r_{14} - y_1r_{34},&
\\
&H_{Y_3} = q_3 - y_4r_{12} + y_2r_{14} - y_1r_{24}, \qquad&&
H_{Y_4} = q_4 + y_3r_{12} - y_2r_{13} + y_1r_{23}.&
\end{alignat*}

The constrained Hamiltonian equation is given by
\begin{gather}
\begin{cases}
\dot{z} = u_1y_1 + u_2y_2 + u_3y_3 + u_4u_4,
\\
\dot{x}_1 = u_1, \quad \dot{x}_2 = u_2, \quad \dot{x}_3 = u_3, \quad \dot{x}_4 = u_4, \quad
\dot{y}_1 = v_1, \quad \dot{y}_2 = v_2, \quad \dot{y}_3 = v_3, \quad \dot{y}_4 = v_4,
\\
\dot{x}_{12} = - x_2u_1 + x_1u_2 - y_4v_3 + y_3v_4, \quad
\dot{x}_{13} = - x_3u_1 + x_1u_3 + y_4v_2 - y_2v_4,
\\
\dot{x}_{14} = - x_4u_1 + x_1u_4 - y_3v_2 + y_2v_3, \quad
\dot{x}_{23} = - x_3u_2 + x_2u_3 - y_4v_1 + y_1v_4,
\\
\dot{x}_{24} = - x_4u_2 + x_2u_4 + y_3v_1 - y_1v_3, \quad
\dot{x}_{34} = - x_4u_3 + x_3u_4 - y_2v_1 + y_1v_2,
\\
\dot{s} = 0, \ \dots \ \dots
\\
\dot{p}_1 = - u_2r_{12} - u_3r_{13} - u_4r_{14}, \quad \dot{p}_2 = u_1r_{12} - u_3r_{23} - u_4r_{24},
\\
\dot{p}_3 = u_1r_{13} + u_2r_{23} - u_4r_{34}, \quad \dot{p}_4 = u_1r_{14} + u_2r_{24} + u_3r_{34},
\\
\dot{q}_1 = - u_1s - v_2r_{34} + v_3r_{24} - v_4r_{23}, \quad \dot{q}_2 = - u_2s + v_1r_{34} - v_3r_{14} + v_4r_{13},
\\
\dot{q}_3 = - u_3s - v_1r_{24} + v_2r_{14} - v_4r_{12}, \quad \dot{q}_2 = - u_4s + v_1r_{23} - v_2r_{13} + v_3r_{12},
\\
\dot{r}_{12} = 0, \quad \dot{r}_{13} = 0, \quad \dot{r}_{14} = 0, \quad \dot{r}_{23} = 0, \quad \dot{r}_{24} = 0, \quad \dot{r}_{34} = 0,
\end{cases}\hspace{-10mm}
 \label{sharp}
\end{gather}
with constraints
\begin{alignat*}{5}
&H_{X_1} = 0, \qquad&& H_{X_2} = 0, \qquad&& H_{X_3} = 0, \qquad&& H_{X_4} = 0,& \\
&H_{Y_1} = 0, \qquad&& H_{Y_2} = 0, \qquad&& H_{Y_3} = 0, \qquad&& H_{Y_4} = 0,&
\end{alignat*}
and $s(t)$, $p_1(t)$, $p_2(t)$, $p_3(t)$, $p_4(t)$, $q_1(t)$, $q_2(t)$, $q_3(t)$, $q_4(t)$, $r_{ij}(t)$ are not all zero for any $t$.

By the constraints, if $s$, $r_{ij}$ are all zero, then $p_i$, $q_j$, $1 \leq i, j \leq 4$ are also zero.
So $s$, $r_{ij}$, $1 \leq i < j \leq 4$ must be not all zero.

\begin{Remark}
In Cartan's model, we have that $s$ and $r_{ij}$ are locally constant by the Hamiltonian equation.
However, we do not use this property in the following arguments.
\end{Remark}

For instance, from the constraint $H_{X_1} = 0$, we have, along any solution curve
by Lemma~\ref{Singular-curve-lemma}, that
\[
0 = \frac{{\rm d}}{{\rm d}t}H_{X_1} = \sum_{i=1}^4 u_iH_{[X_1, X_i]} + \sum_{j=1}^4 v_jH_{[X_1, Y_j]}.
\]

Then similarly from the constraint, we have the following equality in a general form:
\begin{gather*}
\begin{pmatrix}
0 \hspace{-2.3pt}&\hspace{-2.3pt} H_{[X_1, X_2]} \hspace{-2.3pt}&\hspace{-2.3pt} H_{[X_1, X_3]} \hspace{-2.3pt}&\hspace{-2.3pt} H_{[X_1, X_4]} \hspace{-2.3pt}&\hspace{-2.3pt} H_{[X_1, Y_1]} \hspace{-2.3pt}&\hspace{-2.3pt} H_{[X_1, Y_2]} \hspace{-2.3pt}&\hspace{-2.3pt} H_{[X_1, Y_3]} \hspace{-2.3pt}&\hspace{-2.3pt}
H_{[X_1, Y_4]}
\\
H_{[X_2, X_1]} \hspace{-2.3pt}&\hspace{-2.3pt} 0 \hspace{-2.3pt}&\hspace{-2.3pt} H_{[X_2, X_3]} \hspace{-2.3pt}&\hspace{-2.3pt} H_{[X_2, X_4]} \hspace{-2.3pt}&\hspace{-2.3pt} H_{[X_2, Y_1]} \hspace{-2.3pt}&\hspace{-2.3pt} H_{[X_2, Y_2]} \hspace{-2.3pt}&\hspace{-2.3pt} H_{[X_2, Y_3]} \hspace{-2.3pt}&\hspace{-2.3pt}
H_{[X_2, Y_4]}
\\
H_{[X_3, X_1]} \hspace{-2.3pt}&\hspace{-2.3pt} H_{[X_3, X_2]} \hspace{-2.3pt}&\hspace{-2.3pt} 0 \hspace{-2.3pt}&\hspace{-2.3pt} H_{[X_3 X_4]} \hspace{-2.3pt}&\hspace{-2.3pt} H_{[X_3, Y_1]} \hspace{-2.3pt}&\hspace{-2.3pt} H_{[X_3, Y_2]} \hspace{-2.3pt}&\hspace{-2.3pt} H_{[X_3, Y_3]} \hspace{-2.3pt}&\hspace{-2.3pt}
H_{[X_3, Y_4]}
\\
H_{[X_4, X_1]} \hspace{-2.3pt}&\hspace{-2.3pt} H_{[X_4, X_2]} \hspace{-2.3pt}&\hspace{-2.3pt} H_{[X_4 X_3]} \hspace{-2.3pt}&\hspace{-2.3pt} 0 \hspace{-2.3pt}&\hspace{-2.3pt} H_{[X_4, Y_1]} \hspace{-2.3pt}&\hspace{-2.3pt} H_{[X_4, Y_2]} \hspace{-2.3pt}&\hspace{-2.3pt} H_{[X_4, Y_3]} \hspace{-2.3pt}&\hspace{-2.3pt}
H_{[X_4, Y_4]}
\\
H_{[Y_1, X_1]} \hspace{-2.3pt}&\hspace{-2.3pt} H_{[Y_1, X_2]} \hspace{-2.3pt}&\hspace{-2.3pt} H_{[Y_1, X_3]} \hspace{-2.3pt}&\hspace{-2.3pt} H_{[Y_1, X_4]} \hspace{-2.3pt}&\hspace{-2.3pt} 0 \hspace{-2.3pt}&\hspace{-2.3pt} H_{[Y_1, Y_2]} \hspace{-2.3pt}&\hspace{-2.3pt} H_{[Y_1, Y_3]} \hspace{-2.3pt}&\hspace{-2.3pt}
H_{[Y_1, Y_4]}
\\
H_{[Y_2, X_1]} \hspace{-2.3pt}&\hspace{-2.3pt} H_{[Y_2, X_2]} \hspace{-2.3pt}&\hspace{-2.3pt} H_{[Y_2, X_3]} \hspace{-2.3pt}&\hspace{-2.3pt} H_{[Y_2, X_4]} \hspace{-2.3pt}&\hspace{-2.3pt} H_{[Y_2, Y_1]} \hspace{-2.3pt}&\hspace{-2.3pt} 0 \hspace{-2.3pt}&\hspace{-2.3pt} H_{[Y_2, Y_3]} \hspace{-2.3pt}&\hspace{-2.3pt}
H_{[Y_2, Y_4]}
\\
H_{[Y_3, X_1]} \hspace{-2.3pt}&\hspace{-2.3pt} H_{[Y_3, X_2]} \hspace{-2.3pt}&\hspace{-2.3pt} H_{[Y_3, X_3]} \hspace{-2.3pt}&\hspace{-2.3pt} H_{[Y_3, X_4]} \hspace{-2.3pt}&\hspace{-2.3pt} H_{[Y_3, Y_1]} \hspace{-2.3pt}&\hspace{-2.3pt} H_{[Y_3, Y_2]} \hspace{-2.3pt}&\hspace{-2.3pt} 0 \hspace{-2.3pt}&\hspace{-2.3pt}
H_{[Y_3, Y_4]}
\\
H_{[Y_4, X_1]} \hspace{-2.3pt}&\hspace{-2.3pt} H_{[Y_4, X_2]} \hspace{-2.3pt}&\hspace{-2.3pt} H_{[Y_4, X_3]} \hspace{-2.3pt}&\hspace{-2.3pt} H_{[Y_4, X_4]} \hspace{-2.3pt}&\hspace{-2.3pt} H_{[Y_4, Y_1]} \hspace{-2.3pt}&\hspace{-2.3pt} H_{[Y_4, Y_2]} \hspace{-2.3pt}&\hspace{-2.3pt} H_{[Y_4, Y_3]} \hspace{-2.3pt}&\hspace{-2.3pt} 0
\hspace{-1.5pt}\end{pmatrix} \hspace{-3.5pt}
\begin{pmatrix}
u_1
\\
u_2
\\
u_3
\\
u_4
\\
v_1
\\
v_2
\\
v_3
\\
v_4
\end{pmatrix} \!=\!
\begin{pmatrix}
0
\\
0
\\
0
\\
0
\\
0
\\
0
\\
0
\\
0
\end{pmatrix}\!.
\end{gather*}
Explicitly, we have in fact
\begin{gather}
\begin{pmatrix}
\hphantom{-}0 & \hphantom{-}2r_{12} & \hphantom{-}2r_{13} & 2r_{14} & -s & \hphantom{-}0 & \hphantom{-}0 & \hphantom{-}0
\\
-2r_{12} & \hphantom{-}0 & \hphantom{-}2r_{23} &2r_{24} & \hphantom{-}0 & -s & \hphantom{-}0 & \hphantom{-}0
\\
-2r_{13} & -2r_{23} & \hphantom{-}0 & 2r_{34} & \hphantom{-}0 & \hphantom{-}0 & -s & \hphantom{-}0
\\
-2 r_{14} & -2r_{24} & -2r_{34} & 0 & \hphantom{-}0 & \hphantom{-}0 & \hphantom{-}0 & -s
\\
\hphantom{-}s & \hphantom{-}0 & \hphantom{-}0 & 0 & \hphantom{-}0 & \hphantom{-}2r_{34} & -2r_{24} & \hphantom{-}2r_{23}
\\
\hphantom{-}0 & \hphantom{-}s & \hphantom{-}0 & 0 & -2r_{34} & \hphantom{-}0 & \hphantom{-}2r_{14} & -2r_{13}
\\
\hphantom{-}0 & \hphantom{-}0 & \hphantom{-}s & 0 & \hphantom{-}2r_{24} & -2r_{14} & \hphantom{-}0 & \hphantom{-}2r_{12}
\\
\hphantom{-}0 & \hphantom{-}0 & \hphantom{-}0 & s & -2r_{23} & \hphantom{-}2r_{13} & -2r_{12} & \hphantom{-}0
\end{pmatrix}
\begin{pmatrix}
u_1
\\
u_2
\\
u_3
\\
u_4
\\
v_1
\\
v_2
\\
v_3
\\
v_4
\end{pmatrix} =
\begin{pmatrix}
0
\\
0
\\
0
\\
0
\\
0
\\
0
\\
0
\\
0
\end{pmatrix}. \label{star}
\end{gather}
Equivalently, we have
\begin{gather}
\begin{pmatrix}
-v_1 & \hphantom{-}2u_2 & \hphantom{-}2u_3 & \hphantom{-}2u_4 & \hphantom{-}0 & \hphantom{-}0 & \hphantom{-}0
\\
-v_2 & -2u_1 & \hphantom{-}0 & \hphantom{-}0 & \hphantom{-}2u_3 & \hphantom{-}2u_4 & \hphantom{-}0
\\
-v_3 & \hphantom{-}0 & -2u_1 & \hphantom{-}0 & -2u_2 & \hphantom{-}0 & \hphantom{-}2u_4
\\
-v_4 & \hphantom{-}0 & \hphantom{-}0 & -2u_1 & \hphantom{-}0 & -2u_2 & -2u_3
\\
\hphantom{-}u_1 & \hphantom{-}0 & \hphantom{-}0 & \hphantom{-}0 & \hphantom{-}2v_4 & -2v_3 & \hphantom{-}2v_2
\\
\hphantom{-}u_2 & \hphantom{-}0 & -2v_4 & \hphantom{-}2v_3 & \hphantom{-}0 & \hphantom{-}0 & -2v_1
\\
\hphantom{-}u_3 & \hphantom{-}2v_4 & \hphantom{-}0 & -2v_2 & \hphantom{-}0 & \hphantom{-}2v_1 & \hphantom{-}0
\\
\hphantom{-}u_4 & -2v_3 & \hphantom{-}2v_2 & \hphantom{-}0 & -2v_1 & \hphantom{-}0 & \hphantom{-}0
\end{pmatrix}
\begin{pmatrix}
s
\\
r_{12}
\\
r_{13}
\\
r_{14}
\\
r_{23}
\\
r_{24}
\\
r_{34}
\end{pmatrix} =
\begin{pmatrix}
0
\\
0
\\
0
\\
0
\\
0
\\
0
\\
0
\\
0
\end{pmatrix}.\label{star-star}
\end{gather}

Write \eqref{star} as
\[
\begin{pmatrix}
A_{11} & -sI
\\
sI & A_{22}
\end{pmatrix}
\begin{pmatrix}
u
\\
v
\end{pmatrix} =
\begin{pmatrix}
0
\\
0
\end{pmatrix},
\]
where $u = {}^t(u_1, u_2, u_3, u_4)$, $v = {}^t(v_1, v_2, v_3, v_4)$ and $I$ is the $4\times 4$ unit matrix.
We denote by~$A$ the skew-symmetric $8\times 8$ matrix
\smash{$\bigl(
\begin{smallmatrix}
A_{11} & -sI
\\
sI & A_{22}
\end{smallmatrix}
\bigr)$}
and by $U$ the $8\times 7$ matrix which appeared in \eqref{star} and~\eqref{star-star}, respectively.

Then the condition \eqref{star} is equivalent to that
$A_{11}u = sv$, $A_{22}v = -su$. Note that $\det(A_{11}) = \det(A_{22}) = \{4(r_{12}r_{34} - r_{13}r_{24} + r_{14}r_{23})\}^2$
and that
$A_{11}A_{22} = A_{22}A_{11} = -4(r_{12}r_{34} - r_{13}r_{24} + r_{14}r_{23})I$. Then
the condition \eqref{star} implies that
\[
\bigl\{ s^2 - 4(r_{12}r_{34} - r_{13}r_{24} + r_{14}r_{23})\bigr\}u = 0, \qquad \bigl\{ s^2 - 4(r_{12}r_{34} - r_{13}r_{24} + r_{14}r_{23})\bigr\}v = 0.
\]
Therefore, if $(u, v) \not= (0, 0)$, then we have
\[
s^2 - 4(r_{12}r_{34} - r_{13}r_{24} + r_{14}r_{23}) = 0.
\]
Suppose $s \not= 0$. Then, since $A_{11}$ is skew-symmetric, we have that
${}^tu\cdot v = \smash{\frac{1}{s}\,{}^tu\cdot(A_{11}u)} = \smash{\frac{1}{s}\,({}^tuA_{11})\cdot u} = \smash{\frac{1}{s}\,{}^t({}^tA_{11}u)u} = - \smash{\frac{1}{s}\,{}^t(A_{11}u)u} = - {}^tv\cdot u = - {}^tu\cdot v$.
Therefore, we have that
\[
{}^tu \cdot v = u_1v_1 + u_2v_2 + u_3v_3 + u_4v_4 = 0.
\]
Suppose $s = 0$. Then $A_{11}u = 0$ and $A_{22}v = 0$. Note that $A_{11}A_{22} = A_{22}A_{11} = 0$.
Since~$A_{11}$ and~$A_{22}$ are non-zero and skew-symmetric, we have $\rank(A_{11}) = 2$, $\rank(A_{22}) = 2$, and
therefore ${\rm Ker}(A_{11}) = {\mbox{\rm Im}}(A_{22})$ and ${\mbox{\rm Im}}(A_{11}) = {\rm Ker}(A_{22})$.
Then we have $u = A_{22}\tilde{u}$ and $v = A_{11}\tilde{v}$ for some~$\tilde{u}$,~$\tilde{v}$, and thus
${}^tu\cdot v = {}^t(A_{22}\tilde{u})\cdot A_{11}\tilde{v} = {}^t\tilde{u}\,{}^t\!A_{22}A_{11}\tilde{v} = - {}^t\tilde{u}A_{22}A_{11}\tilde{v} = 0$.

\begin{Proposition}
The singular velocity cone $\SVC(D)$ of Cartan's model $D$ is given by
\[
\SVC(D) = \Biggl\{ \sum_{i=1}^4 u_iX_i + \sum_{j=1} v_jY_j \,
\Bigg\vert\, u_1v_1 + u_2v_2 + u_3v_3 + u_4v_4 = 0\Biggr\}.
\]
\end{Proposition}

\begin{proof}
That $\SVC(D)$ is contained in the right hand side is already shown. Let us show the converse inclusion.
All columns of the $8\times 7$ matrix $U$ which appeared in \eqref{star-star} are
null and orthogonal to each other with respect to the metric ${}^tu\cdot v =
u_1v_1 + u_2v_2 + u_3v_3 + u_4v_4$ on $\KK^8$. Note that the metric is non-degenerate for $\KK = \C$
and is of signature $(4, 4)$ if $\KK = \R$.
In any case we have that $\rank(U) \leq 4 < 7$, because the subspace generated by all columns of $U$
is a~null space in $\KK^8$ with respect to the metric ${}^tu\cdot v$.
Hence, for any non-zero constant vector~$(u, v)$ with~${{}^tu\cdot v = 0}$, there exists $(s, r_{ij}) \not= 0$ such that \eqref{star-star} holds, and therefore that~\eqref{star} holds. Thus we see that, given non-zero $(u, v)$ with ${}^tu\cdot v = 0$,
there exist constants $s$, $p_i$, $1 \leq i \leq 4$, $q_j$, $1 \leq j \leq 4$, $r_{ij}$, $1 \leq i < j \leq 4$, which are not all zero,
and functions $x_i$, $1 \leq i \leq 4$, $y_j$, $1 \leq j \leq 4$, $x_{ij}$, $1 \leq i < j \leq 4$
such that the linear ordinary differential equation \eqref{sharp} for singular curves is satisfied.
Thus we see the required equality.
\end{proof}

We define a quadratic form $Q$ on $\KK^8$ and $R$ on $\KK^7$, respectively, by
\[
Q(u, v) := u_1v_1 + u_2v_2 + u_3v_3 + u_4v_4, \qquad R(s, r_{ij}) := s^2 - 4(r_{12}r_{34} - r_{13}r_{24} + r_{14}r_{23}).
\]
The quadratic form $Q$ induces the bi-linear form
\[
\bigl((u, v), \bigl(u', v'\bigr)\bigr) =
\frac{1}{2}\bigl( u_1v_1' + v_1u_1' + u_2v_2' + v_2u_2' + u_3v_3' + v_3u_3' + u_4v_4' + v_4u_4'\bigr)
\]
on $\KK^8 \times \KK^8$. Moreover, the quadratic form $R$ induces the bilinear form
\[
\bigl(\bigl(s, r_{ij}\bigr) \vert \bigl(s', r_{ij}'\bigr)\bigr) = ss' - 2\bigl( r_{12}r_{34}' + r_{34}r_{12}' - r_{13}r_{24}' + r_{24}r_{13}' + r_{14}r_{23}' + r_{23}r_{14}'\bigr)
\]
on $\KK^7\times \KK^7$.

\begin{Corollary}
\label{Conformal-metrics}
The distribution $D \subset T\KK^{15}$ has the canonical non-degenerate metric $(\, ,\, )$ for ${\KK = \C}$ and
the canonical conformal $(4, 4)$-metric $(\, ,\, )$ for $\KK = \R$.
The distribution ${D^\perp \subset T^*\KK^{15}}$ has the canonical non-degenerate metric~$(\, \vert \, )$ for $\KK = \C$ and
the canonical conformal $(4, 3)$-metric~$(\, \vert \, )$ for $\KK = \R$.
\end{Corollary}

\begin{Remark}
\label{Remark-representation-theory}
Let $G = F_{4(4)}$, $P = P_{\alpha_4}$, the parabolic subgroup of $F_{4(4)}$
corresponding to the root~$\alpha_4$,
$X = G/P_{\alpha_4} = \OO' P^2_0$, that is
the hyperplane section of the split Cayley projective space $\OO' P^2$ and $H = \Spin(4, 3)$.
Then we have the decomposition $TG = T_1\oplus T_2$ into $H$-modules, where~$T_1$ (resp.\ $T_2$)
is regarded as the $8$-dimensional spin representation of $\Spin(4, 3)$; $T_1 \cong \OO'$,
(resp.\ the $7$-dimensional vector representation; $T_2 \cong {\mbox{\rm {Im}}}\OO'$).
Moreover, the closed $H$-orbit $Y_1 \subset \PP(T_1)$ (resp.\ $Y_2 \subset \PP(T_2)$)
is a $6$-dimensional quadric (resp.\ is a $5$-dimensional quadric)
with a conformal structure of type $(3, 3)$ (resp.\ of type $(3, 2)$) (see \cite[Section~6.3]{LM}). See
also~\cite[Section~2]{LM} and \cite{BE, IL} for general constructions in simple Lie algebras.

Consider the Clifford algebra ${\mbox{\rm{Cl}}}(4,3) \supset T_1$.
Let ${\mathcal N}$ be the totality of $3$-dimensional null subspaces in $T_1$.
We set $N_s := \{ z \in T_2 \mid z(s) = 0\}$ for $s \in T_1$.
If $N_s \in {\mathcal N}$, $s$ is called a pure spinor. Denote by ${\rm PS}(4, 3)$
the set of pure spinors and by $\PP({\rm PS}(4, 3))$ its projectivisation. Then the correspondence $[s] \in \PP({\rm PS}(4, 3)) \mapsto N_s \in {\mathcal N}$ turns to be an isomorphism. See, for instance, \cite[p.~241 and~p.~283]{Harvey}.

Now in the left hand side of the equality \eqref{star} in our argument in this section,
the action $\uuu = {}^t(u, v) \mapsto A\uuu$
corresponds to the spinor representation of $T_2 \subset {\mbox{\rm Cl}}(4, 3)$ to $T_1$.
Moreover, we see that the set $D$ of solutions
$\uuu$ to the equation $A\uuu = \0$
is exactly equal to the set ${\rm PS}(4, 3)$ of pure spinors.
Thus we see that $D = T_1$ and that {$\SVC(D) \cong \widehat{Y}_1 \cong {\rm PS}(4,3)$}.
Therefore, invariant cone $\widehat{Y}_1$ is constructed from $D = T_1$ algebraically from the viewpoint of
representation theory. Further
$D^\perp = (TX/T_1)^* = T_2^* (\subset T^*X)$ has the $H$-invariant $(4, 3)$-metric.
In this paper, we have characterised these objects known in representation theory by using singular curves
from the viewpoint of geometric control theory.
\end{Remark}

\section{Null flags associated to abnormal bi-extremals}
\label{Null-flags}

We continue to analyse the equation \eqref{star-star} appeared in the previous section.
Recall the $8\times7$ matrix~$U$ which appeared in \eqref{star-star}.
Write \smash{$U = \bigl( \begin{smallmatrix}U' \\ U''\end{smallmatrix}\bigr)$} using $4\times 7$ matrices $U'$, $U''$.
Then we have that
\[
{}^tU U = \bigl({}^tU'' \, {}^t U'\bigr)\begin{pmatrix}U' \\ U''\end{pmatrix}
=
\left(
\begin{array}{ccccccc}
-2Q & 0 & \hphantom{-}0 & 0 & 0 & \hphantom{-}0 & 0
\\
\hphantom{-}0 & 0 & \hphantom{-}0 & 0 & 0 & \hphantom{-}0 & 4Q
\\
\hphantom{-}0 & 0 & \hphantom{-}0 & 0 & 0 & -4Q & 0
\\
\hphantom{-}0 & 0 & \hphantom{-}0 & 0 & 4Q & \hphantom{-}0 & 0
\\
\hphantom{-}0 & 0 & \hphantom{-}0 & 4Q & 0 & \hphantom{-}0 & 0
\\
\hphantom{-}0 & 0 & -4Q & 0 & 0 & \hphantom{-}0 & 0
\\
\hphantom{-}0 & 4Q & \hphantom{-}0 & 0 & 0 & \hphantom{-}0 & 0
\end{array}
\right),
\]
where $Q = u_1v_1 + u_2v_2 + u_3v_3 + u_4v_4$. Note that $\det\bigl({}^tU U\bigr) = 2^{13} Q^8$.

If $Q \not= 0$, then $\rank(U) = 7$. If $Q = 0$, then, since
${}^tU U = O$, regarding $U \colon \KK^7 \to \KK^8$ and
${}^tU \colon \KK^8 \to \KK^7$,
we have that ${\mbox{\rm Im}}(U) \subseteq {\rm Ker}\bigl({}^tU\bigr)$,
so that $\rank(U) \leq 8 - \rank(U)$. Thus we have $\rank(U) \leq 4$ again.
Moreover, if we set
$R = s^2 - 4(r_{12}r_{34} - r_{13}r_{24} + r_{14}r_{23})$, then
we have that if $(u, v) \not= (0, 0)$ and $Q = 0$, then ${\rm Ker}(U) \subseteq R^{-1}(0)$.
So we have ${\rm Ker}(U)$ is a null subspace for the non-degenerate
metric $R'$ induced by the quadratic form $R$ and that
$\dim{\rm Ker}(U) \leq 3$. Thus we have, in fact, $\rank(U) = 4$ and $\dim{\rm Ker}(U) = 3$, if $(u, v) \not= (0, 0)$ and $Q = 0$.
Therefore, we observe that, for any (null) line in $Q^{-1}(0)$, there corresponds a null $3$-pace in $R^{-1}(0)$.
Conversely, for any null $3$-space in $R^{-1}(0)$, there corresponds a null line in $Q^{-1}(0)$.
However, for any null line in $R^{-1}(0)$, naturally there corresponds,
not a null $3$-space, but a null $4$-space in $Q^{-1}(0)$
by the equation \eqref{star}, since, on $R^{-1}(0) \setminus \{ 0\}$, we see $\det(A_{11}) \not= 0$ and the matrix
$A$ is of rank $4$.

In fact we have

\begin{Lemma}
To any null-flag $\Lambda_1 \subset \Lambda_2 \subset \Lambda_3 \subset R^{-1}(0)$
for $R = s^2 - 4(r_{12}r_{34} - r_{13}r_{24} + r_{14}r_{23})$, where $\dim(\Lambda_i) = i$, $i = 1, 2, 3$,
there corresponds uniquely, by the equation \eqref{star}, a null-flag $V_1 \subset V_2 \subset V_4 \subset Q^{-1}(0)$ for
$Q = u_1v_1 + u_2v_2 + u_3v_3 + u_4v_4$, where $\dim(V_k) = k$, $k = 1, 2, 4$.
\end{Lemma}

\begin{proof}
The conformal orthogonal group ${\rm CO}(R)$ of the quadratic form $R$ acts transitively
on the null Grassmannian $\{(\Lambda_1, \Lambda_2, \Lambda_3)\}$ on the metric space
$D^\perp_m \cong \R^{4,3}$, $m \in \KK^{15}$ defined by $R$.
We take the basis of $D_m^\perp$:
\smash{$\varepsilon_1 = \frac{\pa}{\pa s}$}, \smash{$\varepsilon_2 = \frac{\pa}{\pa r_{12}}$},
\smash{$\varepsilon_3 = \frac{\pa}{\pa r_{13}}$},
\smash{$\varepsilon_4 = \frac{\pa}{\pa r_{14}}$},
\smash{$\varepsilon_5 = \frac{\pa}{\pa r_{23}}$},
\smash{$\varepsilon_6 = \frac{\pa}{\pa r_{24}}$},
\smash{$\varepsilon_7 = \frac{\pa}{\pa r_{34}}$}.
Then the representation matrix of the $(4, 3)$-metric on $R$ becomes
\[
\begin{pmatrix}
1 & \hphantom{-}0 & 0 & \hphantom{-}0 & \hphantom{-}0 & 0 & \hphantom{-}0
\\
0 & \hphantom{-}0 & 0 & \hphantom{-}0 & \hphantom{-}0 & 0 & -2
\\
0 & \hphantom{-}0 & 0 & \hphantom{-}0 & \hphantom{-}0 & 2 & \hphantom{-}0
\\
0 & \hphantom{-}0 & 0 & \hphantom{-}0 & -2 & 0 & \hphantom{-}0
\\
0 & \hphantom{-}0 & 0 & -2 & \hphantom{-}0 & 0 & \hphantom{-}0
\\
0 & \hphantom{-}0 & 2 & \hphantom{-}0 & \hphantom{-}0 & 0 & \hphantom{-}0
\\
0 & -2 & 0 & \hphantom{-}0 & \hphantom{-}0 & 0 & \hphantom{-}0
\end{pmatrix}
\]

Then we set
the base point $\bigl(\Lambda_1^0, \Lambda_2^0, \Lambda_3^0\bigr)$ of the null flag manifold
${\mathcal F}'$, where
\[
\Lambda_1^0 = \langle \varepsilon_2\rangle, \qquad
\Lambda_2^0 = \langle \varepsilon_2, \varepsilon_3\rangle, \qquad
\Lambda_3^0 = \langle \varepsilon_2, \varepsilon_3, \varepsilon_4\rangle,
\]

We take the frame
\begin{align*}
f_1 & = z_{11}\varepsilon_1 + \varepsilon_2 + z_{13} \varepsilon_3 + z_{14} \varepsilon_4 + z_{15} \varepsilon_5 +
z_{16} \varepsilon_6 + z_{17} \varepsilon_7,
\\
f_2 & = z_{21}\varepsilon_1 + \hphantom{\varepsilon_2 + z_{13}} \varepsilon_3 + z_{24} \varepsilon_4 + z_{25} \varepsilon_5 + z_{26} \varepsilon_6 + z_{27} \varepsilon_7,
\\
f_3 & = z_{31}\varepsilon_1 + \hphantom{\varepsilon_2 + z_{13} \varepsilon_3 + z_{14}} \varepsilon_4 + z_{35} \varepsilon_5 +
z_{36} \varepsilon_6 + z_{37} \varepsilon_7,
\end{align*}
associated to a (not necessarily null) flag $(\Lambda_1, \Lambda_2, \Lambda_3)$
with $\Lambda_1 = \langle f_1\rangle$, $\Lambda_2 = \langle f_1, f_2\rangle$ and $\Lambda_3 = \langle f_1, f_2, f_3\rangle$ in a neighbourhood of the base point $\bigl(\Lambda_1^0, \Lambda_2^0, \Lambda_3^0\bigr)$.

Then the condition that $(\Lambda_1, \Lambda_2, \Lambda_3)$ is a null flag is equivalent to that
\begin{gather*}
(f_1 \vert f_1) = z_{11} - 4z_{17} + 4z_{13}z_{16} - 4z_{14}z_{15} = 0,
\\
(f_1 \vert f_2) = z_{11}z_{21} - 2z_{27} + 2z_{13}z_{26} - 2z_{14}z_{25} - 2z_{15}z_{24} + 2z_{16} = 0,
\\
(f_1 \vert f_3) = z_{11}z_{31} - 2z_{37} + 2z_{13}z_{36} - 2z_{14}z_{35} - 2z_{15} = 0,
\\
(f_2 \vert f_2) = z_{21}^2 + 4z_{26} - 4z_{24}z_{25} = 0,
\\
(f_2 \vert f_3) = z_{21}z_{31} + 2z_{36} - 2z_{24}z_{35} - 2z_{25} = 0,
\\
(f_3 \vert f_3) = z_{31}^2 - 4z_{35} = 0.
\end{gather*}
Thus the null flag manifold ${\mathcal F}'$ has a system of local coordinates $(z_{11}, z_{13}, z_{14}, z_{15}, z_{16}, z_{21}, z_{24},\allowbreak z_{25}, z_{31})$ and $\dim {\mathcal F}' = 9$.
For $\Lambda_1 = \langle f_1\rangle$, the equation \eqref{star} is equivalent to that
\begin{alignat*}{3}
&2u_2 + 2z_{13}u_3 + 2z_{14}u_4 - z_{11}v_1 = 0,
\qquad&&
-2u_1 + 2z_{15}u_3 + 2z_{16}u_4 - z_{11}v_2 = 0,&
\\
&-2z_{13}u_1 - 2z_{15}u_2 + 2z_{17}u_4 - z_{11}v_3 = 0,
\qquad&&
-2z_{14}u_1 - 2z_{16}u_2 - 2z_{17}u_3 - z_{11}v_4 = 0,&
\\
&z_{11}u_1 + 2z_{17}v_2 - 2z_{16}v_3 + 2z_{15}v_4 = 0,
\qquad&&
z_{11}u_2 - 2z_{17}v_1 + 2z_{14}v_3 - 2z_{13}v_4 = 0,&
\\
&z_{11}u_3 + 2z_{16}v_1 - 2z_{14}v_2 + 2v_4 = 0,
\qquad&&
z_{11}u_4 - 2z_{15}v_1 + 2z_{13}v_2 - 2v_3 = 0,&
\end{alignat*}
and, in fact, to that
\begin{alignat*}{3}
&u_1 = z_{15}u_3 + z_{16}u_4 - \frac{1}{2}z_{11}v_2, \qquad&&
u_2 = -z_{13}u_3 - z_{14}u_4 + \frac{1}{2}z_{11}v_1,&
\\
&v_3 = \frac{1}{2}z_{11}u_4 - z_{15}v_2 + z_{13}v_2, \qquad&&
v_4 = - \frac{1}{2}z_{11}u_3 - z_{16}v_1 + z_{14}v_2.&
\end{alignat*}
Thus we see that the solutions form a null $4$-space $V_4$ in $D_m \cong \KK^8$
\[
V_4 = \left\langle
\begin{matrix}
{}^t\bigl(z_{15}, - z_{13}, 1, 0, 0, 0, 0, - \frac{1}{2}z_{11}\bigr), &
{}^t\bigl(z_{16}, - z_{14}, 0, 1, 0, 0, \frac{1}{2}z_{11}, 0\bigr)
\\
{}^t\bigl(0, \frac{1}{2}z_{11}, 0, 0, 1, 0, - z_{15}, - z_{16}\bigr), &
{}^t\bigl(-\frac{1}{2}z_{11}, 0, 0, 0, 0, 1, z_{13}, z_{14}\bigr)
\end{matrix}
\right\rangle
\]
via the frame $(X_1, X_2, X_3, X_4, Y_1, Y_2, Y_3, Y_4)$.
For $\Lambda_2 = \langle f_1, f_2\rangle$, we get the equation \eqref{star} for $f_1$ as above and, in addition,
the equation \eqref{star} applied to $f_2$,
\begin{alignat*}{3}
&2u_3 + 2z_{24}u_4 - z_{21}v_1 = 0,
\qquad&&
- 2z_{25}u_3 + 2z_{26}u_4 - z_{21}v_2 = 0,&
\\
&-2u_1 - 2z_{25}u_2 + 2z_{27}u_4 - z_{21}v_3 = 0,
\qquad&&
-2z_{24}u_1 - 2z_{26}u_2 - 2z_{27}u_3 - z_{21}v_4 = 0,&
\\
&z_{21}u_1 + 2z_{27}v_2 - 2z_{26}v_3 + 2z_{25}v_4 = 0,
\qquad&&
z_{21}u_2 - 2z_{27}v_1 + 2z_{24}v_3 - 2v_4 = 0,&
\\
&z_{21}u_3 + 2z_{26}v_1 - 2z_{24}v_2 = 0,
\qquad&&
z_{21}u_4 - 2z_{25}v_1 + 2v_2 = 0.&
\end{alignat*}
Then, by the two systems of linear equations for $f_1$ and $f_2$, we have
\begin{gather*}
u_1 = \biggl(-z_{15}z_{24} + z_{16} + \frac{1}{4}z_{11}z_{21}\biggr)u_4 + \biggl(\frac{1}{2}z_{15}z_{21} - \frac{1}{2}z_{11}z_{25}\biggr)v_1
\\
u_2 = (z_{13}z_{24} - z_{14})u_4 + \biggl(-\frac{1}{2}z_{13}z_{21} + \frac{1}{2}z_{11}\biggr)v_1,
\\
u_3 = - z_{24}u_4 + \frac{1}{2}z_{21}v_1,
\\
v_2 = - \frac{1}{2}z_{21}u_4 + \frac{1}{2}z_{21}v_1
\\
v_3 = \biggl(\frac{1}{2}z_{11} - \frac{1}{2}z_{13}z_{21}\biggr)u_4 + (- z_{15} + z_{13}z_{25})v_1,
\\
v_4 = \biggl(\frac{1}{2}z_{11}z_{24} - \frac{1}{2}z_{14}z_{21}\biggr)u_4 + \biggl(- \frac{1}{4}z_{11}z_{21} - z_{16} + z_{14}z_{25}\biggr)v_1.
\end{gather*}
Thus we see $\Lambda_2$ corresponds to the null $2$-plane $V_2$ in $D_m \cong \KK^8$, by \eqref{star},
generated by two vectors
\begin{gather*}
{}^t\bigl(- z_{15}z_{24} \hspace{-0.6pt}+\hspace{-0.6pt} z_{16}\hspace{-0.6pt} +\hspace{-0.6pt} \tfrac{1}{4}z_{11}z_{21}, z_{13}z_{24} \hspace{-0.6pt}-\hspace{-0.6pt} z_{14}, - z_{24}, 1, 0, - \tfrac{1}{2}z_{21},
\tfrac{1}{2}z_{11}\hspace{-0.6pt}-\hspace{-0.6pt} \tfrac{1}{2}z_{13}z_{21}, \tfrac{1}{2}z_{11}z_{24} \hspace{-0.6pt}-\hspace{-0.6pt} \tfrac{1}{2}z_{14}z_{21}\bigr),
\\
{}^t\bigl(\tfrac{1}{2}z_{15}z_{21} \hspace{-0.6pt}-\hspace{-0.6pt} \tfrac{1}{2}z_{11}z_{25}, - \tfrac{1}{2}z_{13}z_{21} \hspace{-0.6pt}+\hspace{-0.6pt} \tfrac{1}{2}z_{11},
\tfrac{1}{2}z_{21}, 0, 1, z_{25}, - z_{15} \hspace{-0.6pt}+\hspace{-0.6pt} z_{13}z_{25}, -\tfrac{1}{4}z_{11}z_{21} \hspace{-0.6pt}-\hspace{-0.6pt} z_{16} \hspace{-0.6pt}+\hspace{-0.6pt} z_{14}z_{25}\bigr).
\end{gather*}
For $\Lambda_3 = \langle f_1, f_2, f_3\rangle$, we obtain the additional condition \eqref{star} applied to $f_3$,
which is given by
\begin{alignat*}{3}
&2u_4 - z_{31}v_1 = 0,
\qquad&&
2z_{35}u_3 + 2z_{36}u_4 - z_{31}v_2 = 0,&
\\
&-2 z_{35}u_2 + 2z_{37}u_4 - z_{31}v_3 = 0,
\qquad&&
-2u_1 - 2z_{36}u_2 - 2z_{37}u_3 - z_{31}v_4 = 0,&
\\
&z_{31}u_1 + 2z_{37}v_2 - 2z_{36}v_3 + 2z_{35}v_4 = 0,
\qquad&&
z_{31}u_2 - 2z_{37}v_1 + 2v_3 = 0,&
\\
&z_{31}u_3 + 2z_{36}v_1 - 2v_2 = 0,
\qquad&&
z_{31}u_4 - 2z_{35}v_1 = 0.&
\end{alignat*}
Then, from the conditions \eqref{star} for $f_1$, $f_2$ and $f_3$, we have
\begin{gather*}
u_1 = \biggl(- \frac{1}{2}z_{11}z_{25} + \frac{1}{2}z_{16}z_{31} + \frac{1}{8}z_{11}z_{21}z_{31}
- \frac{1}{2}z_{15}z_{24}z_{31} + \frac{1}{2}z_{15}z_{21}\biggr)v_1,
\\
u_2 = \biggl(\frac{1}{2}z_{11} - \frac{1}{2}z_{13}z_{21} - \frac{1}{2}z_{14}z_{31} + \frac{1}{2}z_{13}z_{24}z_{31}\biggr)v_1,
\\
u_3 = \biggl(\frac{1}{2}z_{21} - \frac{1}{2}z_{24}z_{31}\biggr)v_1,
\\
u_4 = \frac{1}{2}z_{31}v_1,
\\
v_2 = \biggl(z_{25} -\frac{1}{4}z_{21}z_{31}\biggr)v_1,
\\
v_3 = \biggl(- z_{15} + \frac{1}{4}z_{11}z_{31} + z_{13}z_{25} - \frac{1}{4}z_{13}z_{21}z_{31}\biggr)v_1,
\\
v_4 = \biggl(- z_{16} - \frac{1}{4}z_{11}z_{21} + z_{14}z_{25} + \frac{1}{4}z_{11}z_{24}z_{31}
- \frac{1}{4}z_{14}z_{21}z_{31}\biggr)v_1.
\end{gather*}
Therefore, if we set, by taking $v_1 = 1$,
\[
\eta_1 =
\begin{pmatrix}
- \frac{1}{2}z_{11}z_{25} + \frac{1}{2}z_{16}z_{31} + \frac{1}{8}z_{11}z_{21}z_{31}
- \frac{1}{2}z_{15}z_{24}z_{31} + \frac{1}{2}z_{15}z_{21}
\\
\frac{1}{2}z_{11} - \frac{1}{2}z_{13}z_{21} - \frac{1}{2}z_{14}z_{31} + \frac{1}{2}z_{13}z_{24}z_{31}
\\
\frac{1}{2}z_{21} - \frac{1}{2}z_{24}z_{31}
\\
\frac{1}{2}z_{31}
\\
1
\\
z_{25} -\frac{1}{4}z_{21}z_{31}
\\
- z_{15} + \frac{1}{4}z_{11}z_{31} + z_{13}z_{25} - \frac{1}{4}z_{13}z_{21}z_{31}
\\
- z_{16} - \frac{1}{4}z_{11}z_{21} + z_{14}z_{25} + \frac{1}{4}z_{11}z_{24}z_{31} - \frac{1}{4}z_{14}z_{21}z_{31}
\end{pmatrix},
\]
then we see that $\Lambda_3$ corresponds to the null line $V_1$ generated by $\eta_1$.
Moreover, we set $\eta_2$ by
$
{}^t\bigl(z_{16} + \frac{1}{4}z_{11}z_{21} - z_{15}z_{24}, \,- z_{14} + z_{13}z_{24}, \, - z_{24}, \,1, \,0, \,- \frac{1}{2}z_{21}, \,
\frac{1}{2}z_{11} - \frac{1}{2}z_{13}z_{21}, \,\frac{1}{2}z_{11}z_{24} - \frac{1}{2}z_{14}z_{21}\bigr)
$,
$\eta_3$~by ${}^t(z_{15}, - z_{13}, 1, 0, 0, 0, 0, - \frac{1}{2}z_{11})
$, and $\eta_4$ by
$
{}^t\bigl(-\frac{1}{2}z_{11}, 0, 0, 0, 0, 1, z_{13}, z_{14}\bigr)
$.
Then we have that~$(\eta_1, \eta_2, \eta_3, \eta_4)$ is a frame of $V_4$ satisfying
$V_1 = \langle \eta_1\rangle \subset V_2 = \langle \eta_1, \eta_2\rangle \subset V_4 = \langle \eta_1, \eta_2, \eta_3, \eta_4\rangle$.
\end{proof}

\ber
{\rm
The total null flag bundle $\widetilde{\mathcal F}$ constructed from $D$ which consists of all null flags $V_1 \subset V_2 \subset V_4 \subset D_m \cong \R^{4,4}$, $m \in \KK^{15}$
is of dimension $15 + 11 = 26$.
The total null flag bundle $\widetilde{\mathcal F}'$ constructed from $D^\perp$ which consists of $\Lambda_1 \subset
\Lambda_2 \subset \Lambda_3 \subset D^\perp_m \cong \R^{4,3}$, $m \in \KK^{15}$, is of dimension $15+9 = 24$.
Then we have obtained, as above, the embedding \smash{$\widetilde{\mathcal F}' \to \widetilde{\mathcal F}$} of codimension $2$.
}
\enr

\section{Prolongation of Cartan's model}
\label{Prolongation-of-Cartan's-$(8, 15)$-distribution}

The theory of prolongations and equivalence problems of distributions are established by many authors (see, for instance,
\cite{BCGGG, Montgomery, Morimoto, Tanaka1}).
Here we provide, related to the notion of singular curves of distributions, a way of prolongations
from viewpoints of sub-Riemannian geometry and geometric control theory.

We set, as the prolonged space,
$W = \widetilde{\mathcal F}' \cong \KK^{15}\times {\mathcal F}'$
in $\widetilde{\mathcal F} \cong \KK^{15}\times {\mathcal F}$
by the null flag manifold~${\mathcal F}'$.
Note that $\dim({\mathcal F}') = 9$ and that $\dim(W) = 24$:
$W$ has a local coordinate system
\begin{gather*}
(z,\hspace{-0.2pt} x_1,\hspace{-0.2pt} x_2,\hspace{-0.2pt} x_3,\hspace{-0.2pt} x_4,\hspace{-0.2pt} y_1,\hspace{-0.2pt} y_2,\hspace{-0.2pt} y_3,\hspace{-0.2pt} y_4,\hspace{-0.2pt} x_{12},\hspace{-0.2pt} x_{13},\hspace{-0.2pt} x_{14},\hspace{-0.2pt} x_{23},\hspace{-0.2pt} x_{24},\hspace{-0.2pt} x_{34};\hspace{-0.2pt}
z_{11},\hspace{-0.2pt} z_{13},\hspace{-0.2pt} z_{14},\hspace{-0.2pt} z_{15},\hspace{-0.2pt} z_{16},\hspace{-0.2pt} z_{21},\hspace{-0.2pt}
 z_{24},\hspace{-0.2pt} z_{25},\hspace{-0.2pt} z_{31}).
\end{gather*}
We are going to define and study the canonical distribution $E$ on $W = \widetilde{\mathcal F}'$.

Take any point $w_0 = (m_0, (V_1)_0, (V_2)_0, (V_4)_0)$ of $W$. Then
we define $E_{w_0} \subset T_{w_0}W$ as the set of initial vectors
$
\bigl(m'(0), V_1'(0), V_2'(0), V_4'(0)\bigr)
$
of curves $(m(t), V_1(t), V_2(t), V_4(t)) \colon (\R, 0) \to W$ in~$W$ which satisfy the condition
$m'(t) \in V_1(t)$, $\eta_1'(t) \in V_2(t)$, $\eta_2'(t) \in V_4(t)$ for some (so equivalently for any) framing
$V_1(t) = \langle \eta_1(t)\rangle$, $V_2(t) = \langle \eta_1(t), \eta_2(t)\rangle$.

Now we calculate the canonical distribution $E$ explicitly.
The above condition for $E \subset TW$ reads, at $t = 0$, that
\begin{gather*}
m'(0) = p \eta_1(m(0)), \qquad \eta_1'(0) = q \eta_1(m(0)) + r \eta_2(m(0)),
\\
f_2'(0) = s \eta_1(m(0)) + u \eta_2(m(0)) + v \eta_3(m(0)) + w \eta_4(m(0)),
\end{gather*}
for some $p, q, r, s, u, v, w \in \R$.

By the above second condition
$\eta_1' = q\,\eta_1 + r\,\eta_2$ at $t = 0$, we see $q = 0$ and $\frac{1}{2}z_{31}' = r$ at $t = 0$.
Moreover, after some straightforward calculations, we have
\begin{gather*}
z_{11}' - z_{21}z_{13}' - z_{13}z_{21}' - z_{31}z_{14}' + z_{24}z_{31}z_{13}' + z_{13}z_{31}z_{24}' = 0,
\\
z_{21}' - z_{31}z_{24}' = 0, \qquad z_{25}' - \frac{1}{4}z_{31}^2z_{24}' = 0,
\\
z_{15}' - \frac{1}{4}z_{31}z_{11}' - z_{25}z_{13}' - z_{13}z_{25}' + \frac{1}{4}z_{21}z_{31}z_{13}' +
\frac{1}{4}z_{13}z_{31}z_{21}' = 0,
\\
z_{16}' + \frac{1}{4}z_{21}z_{11}' + \frac{1}{4}z_{11}z_{21}' - z_{25}z_{14}' - z_{14}z_{25}' - \frac{1}{4}z_{24}z_{31}z_{11}' -
\frac{1}{4}z_{11}z_{31}z_{24}' \\
\qquad{} + \frac{1}{4}z_{21}z_{31}z_{14}' + \frac{1}{4}z_{14}z_{31}z_{21}' = 0,
\end{gather*}
at $t = 0$, for the coordinate functions of the curve $\eta_1$ on ${\mathcal F}'$.
By the above third condition
$f_2' = s\,\eta_1 + u\,\eta_2 + v\,\eta_3 + w\,\eta_4$ at $t = 0$, we have that
$s = u = 0$ and that $-z_{24}' = v$, $- \frac{1}{2}z_{21}' = w$ at~$t = 0$.
Moreover, we have that
\[
z_{16}' - z_{24}z_{15}' + \frac{1}{4}z_{21}z_{11}' = 0, \qquad z_{14}' - z_{24}z_{13}' 0, z_{11}' - z_{21}z_{13}' = 0,
\]
at $t = 0$.
In term of differential $1$-forms, the above conditions are reduced to that
\begin{gather*}
{\rm d}z_{11} - z_{21}{\rm d}z_{13} = 0, \qquad {\rm d}z_{21} - z_{31}{\rm d}z_{24} = 0, \\ {\rm d}z_{14} - z_{24}{\rm d}z_{13} = 0, \qquad
{\rm d}z_{25} - \frac{1}{4}z_{31}^2{\rm d}z_{24} = 0,
\\
{\rm d}z_{15} - z_{25}{\rm d}z_{13} = 0, \qquad {\rm d}z_{16} - \biggl(z_{24}z_{25} - \frac{1}{4}z_{21}^2\biggr){\rm d}z_{13} = 0,
\end{gather*}
at $t = 0$.
To get the frame of $E$, we set
\[
\zeta = A\frac{\pa}{\pa z_{11}} + B\frac{\pa}{\pa z_{13}} + C\frac{\pa}{\pa z_{14}} + D\frac{\pa}{\pa z_{15}} +
F\frac{\pa}{\pa z_{16}} + G\frac{\pa}{\pa z_{21}} + H\frac{\pa}{\pa z_{24}} + I\frac{\pa}{\pa z_{25}} + J\frac{\pa}{\pa z_{31}}.
\]
The condition that $\zeta$ belongs to $E$ is given by
\begin{alignat*}{4}
&A - z_{21}B = 0, \qquad&& G - z_{31}H = 0, \qquad&& C - z_{25}B = 0,& \\
&I - \frac{1}{4}z_{31}^2H = 0, \qquad&& D - z_{25}B = 0, \qquad&&
F - \biggl(z_{24}z_{25} - \frac{1}{4}z_{21}^2\biggr)B = 0,&
\end{alignat*}
and thus we have, for some $B, H, J \in \R$,
\begin{align*}
\zeta = B\biggl\{&{} \frac{\pa}{\pa z_{13}} + z_{21}\frac{\pa}{\pa z_{11}} + z_{24}\frac{\pa}{\pa z_{14}} +
z_{25}\frac{\pa}{\pa z_{15}} + \biggl(z_{24}z_{25} - \frac{1}{4}z_{21}^2\biggr)\frac{\pa}{\pa z_{16}}\biggr\}
\vspace{0.3truecm}
\\
&{}
{+}\, H\biggl(\frac{\pa}{\pa z_{24}} + z_{31}\frac{\pa}{\pa z_{21}} + \frac{1}{4}z_{31}^2\frac{\pa}{\pa z_{25}}\biggr) +
J\frac{\pa}{\pa z_{31}},
\end{align*}
at $t = 0$.

Thus, adding the generator which comes from the condition $m'(0) = p\,\eta_1(m(0))$, we have the following lemma.

\begin{Lemma}
\label{Generators-of-E}
We have on the $24$-dimensional
space $W = \R^{15}\times {\mathcal F}'$ with local coordinates $z$, $x_i$, $y_j$, $x_{ij}$, $z_{k\ell}$,
the prolonged distribution $E$ with the system of generators
\begin{gather*}
\zeta_1 = \frac{\pa}{\pa z_{13}} + z_{21}\frac{\pa}{\pa z_{11}} + z_{24}\frac{\pa}{\pa z_{14}} +
z_{25}\frac{\pa}{\pa z_{15}} + \biggl(z_{24}z_{25} - \frac{1}{4}z_{21}^2\biggr)\frac{\pa}{\pa z_{16}},
\\
\zeta_2 = \frac{\pa}{\pa z_{24}} + z_{31}\frac{\pa}{\pa z_{21}} + \frac{1}{4}z_{31}^2\frac{\pa}{\pa z_{25}},
\\
\zeta_3 = \frac{\pa}{\pa z_{31}},
\\
\zeta_4 = \biggl(-\frac{1}{2}z_{11}z_{25} + \frac{1}{2}z_{15}z_{21} + \frac{1}{2}z_{16}z_{31}
+ \frac{1}{8}z_{11}z_{21}z_{31} - \frac{1}{2}z_{15}z_{24}z_{31}\biggr)X_1
\\
\hphantom{\zeta_4 = \biggl(}{}
 + \biggl(\frac{1}{2}z_{11} - \frac{1}{2}z_{13}z_{21} - \frac{1}{2}z_{14}z_{31} + \frac{1}{2}z_{13}z_{24}z_{31}\biggr)X_2
+ \biggl(\frac{1}{2}z_{21} - \frac{1}{2}z_{24}z_{31}\biggr)X_3 + \frac{1}{2}z_{31}X_4
\\
\hphantom{\zeta_4 = \biggl(}{}
 + \ Y_1 + \biggl(z_{25} - \frac{1}{4}z_{21}z_{31}\biggr)Y_2 +
\biggl(-z_{15} + \frac{1}{4}z_{11}z_{31} + z_{13}z_{25} - \frac{1}{4}z_{13}z_{21}z_{31}\biggr)Y_3
\\
\hphantom{\zeta_4 = \biggl(}{}
 + \biggl(-z_{16} - \frac{1}{4}z_{11}z_{21} + z_{14}z_{25} + \frac{1}{4}z_{11}z_{24}z_{31}
- \frac{1}{4}z_{14}z_{21}z_{31}\biggr)Y_4.
\end{gather*}
\end{Lemma}

Note that the vector field $\zeta_4$ in Lemma~\ref{Generators-of-E} is induced from $\eta_1$
obtained in the previous section. We have chosen the above system of generators regarding
the $F_4$-Dynkin diagram (see Remark~\ref{correspondence-generators-roots}).
Now we have the following.

\begin{Lemma}
\label{F_4-bracket-relations}
The growth vector of the distribution $E$ defined in the previous Lemma~$\ref{Generators-of-E}$
is given by $(4, 7, 10, 13, 16, 18, 20, 21, 22, 23, 24)$ and the following bracket relations for the generators
$\zeta_1$, $\zeta_2$, $\zeta_3$, $\zeta_4$ of $E$ given in Lemma~$\ref{Generators-of-E}$:
\begin{alignat*}{8}
&[\zeta_1, \zeta_2] = \zeta_5,\qquad&& [\zeta_1, \zeta_3] = 0,\qquad&& [\zeta_1, \zeta_4] = 0,\qquad&& [\zeta_2, \zeta_3] = \zeta_6,&\\
& [\zeta_2, \zeta_4] = 0,\qquad&& [\zeta_3, \zeta_4] = \zeta_7 \quad \makebox[0pt][l]{\text{in $E^{(2)}$;}}&
\\
&[\zeta_1, \zeta_5] = 0,\qquad&& [\zeta_1, \zeta_6] = \zeta_8,\qquad&& [\zeta_1, \zeta_7] = 0,\qquad&& [\zeta_2, \zeta_5] = 0,&\\
&
[\zeta_2, \zeta_6] = 0,\qquad&&[\zeta_2, \zeta_7] = \zeta_9,\qquad&&[\zeta_3, \zeta_5] = - \zeta_8,\qquad&& [\zeta_3, \zeta_6] = \zeta_{10},&\\
&[\zeta_3, \zeta_7] = 0,\qquad&& [\zeta_4, \zeta_5] = 0,\qquad&& [\zeta_4, \zeta_6] = - \zeta_9,\qquad&& [\zeta_4, \zeta_7] = 0 \quad \makebox[0pt][l]{\text{in $E^{(3)}$;}}&\\
&[\zeta_1, \zeta_8] = 0,\qquad&& [\zeta_1, \zeta_9] = \zeta_{11},\qquad&& [\zeta_1, \zeta_{10}] = \zeta_{12},\qquad&&
[\zeta_2, \zeta_8] = 0,&\\
& [\zeta_2, \zeta_9] = 0,\qquad&& [\zeta_2, \zeta_{10}] = 0,\qquad&& [\zeta_3, \zeta_8] = \zeta_{12},\qquad&&[\zeta_3, \zeta_9] = \zeta_{13},&
\\
&[\zeta_3, \zeta_{10}] = 0,\qquad&& [\zeta_4, \zeta_8] = - \zeta_{11},\qquad&& [\zeta_4, \zeta_9] = 0,\qquad&& [\zeta_4, \zeta_{10}] = - 2\zeta_{13} \quad \makebox[0pt][l]{\text{in $E^{(4)}$;}}&
\\
&[\zeta_1, \zeta_{11}] = 0,\qquad&& [\zeta_1, \zeta_{12}] = 0,\qquad&& [\zeta_1, \zeta_{13}] = \zeta_{14},\qquad&&
[\zeta_2, \zeta_{11}] = 0,& \\
&[\zeta_2, \zeta_{12}] = \zeta_{15},\qquad&& [\zeta_2, \zeta_{13}] = 0,\qquad&& [\zeta_3, \zeta_{11}] = \zeta_{14},\qquad&&[\zeta_3, \zeta_{12}] = 0,
\\
&[\zeta_3, \zeta_{13}] = 0,\qquad&& [\zeta_4, \zeta_{11}] = 0,\qquad&&
[\zeta_4, \zeta_{12}] = -2 \zeta_{14},\qquad&& [\zeta_4, \zeta_{13}] = \zeta_{16} \quad \makebox[0pt][l]{\text{in $E^{(5)}$;}}&
\\
&[\zeta_1, \zeta_{14}] = 0,\qquad&& [\zeta_1, \zeta_{15}] = 0,\qquad&& [\zeta_1, \zeta_{16}] = 0,\qquad&& [\zeta_2, \zeta_{14}] = \zeta_{17},&\\
& [\zeta_2, \zeta_{15}] = 0,\qquad&& [\zeta_2, \zeta_{16}] = 0,\qquad&& [\zeta_3, \zeta_{14}] = 0,\qquad&&[\zeta_3, \zeta_{15}] = 0,&
 \\
&[\zeta_3, \zeta_{16}] = 0,\qquad&& [\zeta_4, \zeta_{14}] = \zeta_{18},\qquad&& [\zeta_4, \zeta_{15}] = -2\zeta_{17},\qquad&&
[\zeta_4, \zeta_{16}] = 0 \quad \makebox[0pt][l]{\text{in $E^{(6)}$;}}&
\\
&[\zeta_1, \zeta_{17}] = 0,\qquad&& [\zeta_1, \zeta_{18}] = 0,\qquad&& [\zeta_2, \zeta_{17}] = 0,\qquad&& [\zeta_2, \zeta_{18}] = \zeta_{19},&\\
&[\zeta_3, \zeta_{17}] = \zeta_{20},\qquad&& [\zeta_3, \zeta_{18}] = 0,\qquad&& [\zeta_4, \zeta_{17}] = \zeta_{19},\qquad&&
 [\zeta_4, \zeta_{18}] = 0 \quad \makebox[0pt][l]{\text{in $E^{(7)}$;}}&
 \\
&[\zeta_1, \zeta_{19}] = 0,\qquad&& [\zeta_1, \zeta_{20}] = 0,\qquad&& [\zeta_2, \zeta_{19}] = 0,\qquad&& [\zeta_2, \zeta_{20}] = 0,&\\
&[\zeta_3, \zeta_{19}] = \zeta_{21},\qquad&& [\zeta_3, \zeta_{20}] = 0,\qquad&&[\zeta_4, \zeta_{19}] = 0,\qquad&&
 [\zeta_4, \zeta_{20}] = \frac{1}{2}\zeta_{21} \quad \makebox[0pt][l]{\text{in $E^{(8)}$;}}&
 \\
&[\zeta_1, \zeta_{19}] = 0,\qquad&& [\zeta_1, \zeta_{20}] = 0,\qquad&& [\zeta_2, \zeta_{19}] = 0,\qquad&& [\zeta_2, \zeta_{20}] = 0,&
 \\
& [\zeta_3, \zeta_{19}] = \zeta_{21},\qquad&& [\zeta_3, \zeta_{20}] = 0,\qquad&&[\zeta_4, \zeta_{19}] = 0,\qquad&&
 [\zeta_4, \zeta_{20}] = \frac{1}{2}\zeta_{21} \quad \makebox[0pt][l]{\text{in $E^{(8)}$;}}&
 \\
&[\zeta_1, \zeta_{21}] = 0,\qquad&& [\zeta_2, \zeta_{21}] = 0,\qquad&& [\zeta_3, \zeta_{21}] = \zeta_{22},\qquad&& [\zeta_4, \zeta_{21}] = 0
 \quad \makebox[0pt][l]{\text{in $E^{(9)}$;}}&
 \\
&[\zeta_1, \zeta_{22}] = 0,\qquad&& [\zeta_2, \zeta_{22}] = \zeta_{23},\qquad&& [\zeta_3, \zeta_{22}] = 0,\qquad&& [\zeta_4, \zeta_{22}] = 0
\quad \makebox[0pt][l]{\text{in $E^{(10)}$;}}&
 \\
&[\zeta_1, \zeta_{23}] = \zeta_{24},\qquad&& [\zeta_2, \zeta_{23}] = 0,\qquad&& [\zeta_3, \zeta_{23}] = 0,\qquad&& [\zeta_4, \zeta_{23}] = 0
 \quad \makebox[0pt][l]{\text{in $E^{(11)} = TW$,}}&
\end{alignat*}
which are calculated explicitly in the proof.
In particular, the distribution $E$ is isomorphic to the $(8, 15)$-distribution on
the quotient space by the parabolic subgroup associated to the root $\alpha_4$ of~$F_4$ in the complex case $($resp.
of $F_{4(4)}$ in the real case$)$.
\end{Lemma}

\begin{Remark}\label{correspondence-generators-roots}
Between the simple roots $\alpha_1$, $\alpha_2$, $\alpha_3$, $\alpha_4$ of $F_4$ (see, for instance, \cite{Bourbaki}) and
the generators $\zeta_1$, $\zeta_2$, $\zeta_3$, $\zeta_4$ of $E$, there exists the correspondence
$\zeta_i \longleftrightarrow -\alpha_i$, $i = 1, 2, 3, 4$,
\begin{alignat*}{3}
&\zeta_5 \longleftrightarrow - (\alpha_1 + \alpha_2), \qquad&&
\zeta_6 \longleftrightarrow -(\alpha_2 + \alpha_3), & \\
&\zeta_7 \longleftrightarrow -(\alpha_3 + \alpha_4), \qquad&&
\zeta_8 \longleftrightarrow -(\alpha_1 + \alpha_2 + \alpha_3), \\
&\zeta_9 \longleftrightarrow -(\alpha_2 + \alpha_3 + \alpha_4), \qquad&&
\zeta_{10} \longleftrightarrow -(\alpha_2 + 2\alpha_3), \\
&\zeta_{11} \longleftrightarrow -(\alpha_1 + \alpha_2 + \alpha_3 + \alpha_4), \qquad&&
\zeta_{12} \longleftrightarrow -(\alpha_1 + \alpha_2 + 2\alpha_3), \\
&\zeta_{13} \longleftrightarrow -(\alpha_2 + 2\alpha_3 + \alpha_4), \qquad&&
\zeta_{14} \longleftrightarrow -(\alpha_1 + \alpha_2 + 2\alpha_3 + \alpha_4), \\
&\zeta_{15} \longleftrightarrow -(\alpha_1 + 2\alpha_2 + 2\alpha_3), \qquad&&
\zeta_{16} \longleftrightarrow -(\alpha_2 + 2\alpha_3 + 2\alpha_4), \\
&\zeta_{17} \longleftrightarrow -(\alpha_1 + \alpha_2 + 2\alpha_3 + \alpha_4), \qquad&&
\zeta_{18} \longleftrightarrow -(\alpha_1 + \alpha_2 + 2\alpha_3 + 2\alpha_4), \\
&\zeta_{19} \longleftrightarrow -(\alpha_1 + 2\alpha_2 + 2\alpha_3 + 2\alpha_4), \qquad&&
\zeta_{20} \longleftrightarrow -(\alpha_1 + 2\alpha_2 + 3\alpha_3 + \alpha_4), \\
&\zeta_{21} \longleftrightarrow -(\alpha_1 + 2\alpha_2 + 3\alpha_3 + 2\alpha_4), \qquad&&
\zeta_{22} \longleftrightarrow -(\alpha_1 + 2\alpha_2 + 4\alpha_3 + 2\alpha_4),\\
&\zeta_{23} \longleftrightarrow -(\alpha_1 + 3\alpha_2 + 4\alpha_3 + 2\alpha_4), \qquad&&
\zeta_{24} \longleftrightarrow -(2\alpha_1 + 3\alpha_2 + 4\alpha_3 + 2\alpha_4).
\end{alignat*}
\end{Remark}

\begin{proof}[Proof of Lemma~\ref{F_4-bracket-relations}]
In fact, we have for the vector fields $\zeta_1$, $\zeta_2$, $\zeta_3$, $\zeta_4$ in Lemma~\ref{Generators-of-E}:
\begin{gather*}
[\zeta_1, \zeta_2] = - \frac{\pa}{\pa z_{14}} - z_{31}\frac{\pa}{\pa z_{11}} - \frac{1}{4}z_{31}^2\frac{\pa}{\pa z_{15}}
+ \biggl(
- z_{25} + \frac{1}{2}z_{21}z_{31} - \frac{1}{4}z_{24}z_{31}^2\biggr)\frac{\pa}{\pa z_{16}} =: \zeta_5,
\\
[\zeta_1, \zeta_3] = 0,\qquad [\zeta_1, \zeta_4] = 0,
\\
[\zeta_2, \zeta_3] = - \frac{\pa}{\pa z_{21}} - \frac{1}{2}z_{31}\frac{\pa}{\pa z_{25}} =: \zeta_6,
\qquad
[\zeta_2, \zeta_4] = 0,
\\
[\zeta_3, \zeta_4] = \biggl(\frac{1}{2}z_{16} + \frac{1}{8}z_{11}z_{21} - \frac{1}{2}z_{15}z_{24}\biggr)X_1 +
\biggl(- \frac{1}{2}z_{14} + \frac{1}{2}z_{13}z_{24}\biggr)X_2 - \frac{1}{2}z_{24}X_3 + \frac{1}{2}X_4
\\
\hphantom{[\zeta_3, \zeta_4] =}{}
 - \frac{1}{4}z_{21}Y_2
+ \biggl(\frac{1}{4}z_{11} - \frac{1}{4}z_{13}z_{21}\biggr)Y_3 + \biggl(\frac{1}{4}z_{11}z_{24} - \frac{1}{4}z_{14}z_{21}\biggr)Y_4 =: \zeta_7.
\end{gather*}
So far, we have \smash{$\rank E^{(2)} = 7$}.

Moreover, we have
\begin{gather*}
[\zeta_1, \zeta_5] = 0,
\qquad
[\zeta_1, \zeta_6] = \frac{\pa}{\pa z_{11}} + \frac{1}{2}z_{31}\frac{\pa}{\pa z_{15}}
+ \biggl(-\frac{1}{2}z_{21} + \frac{1}{2}z_{24}z_{31}\biggr)\frac{\pa}{\pa z_{16}} =: \zeta_8,
\\
[\zeta_1, \zeta_7] = 0, \\
[\zeta_2, \zeta_5] = 0, \qquad [\zeta_2, \zeta_6] = 0,
\\
[\zeta_2, \zeta_7] = \biggl(-\frac{1}{2}z_{15} + \frac{1}{8}z_{11}z_{31}\biggr)X_1 +
\frac{1}{2}z_{13}X_2 - \frac{1}{2}X_3
- \frac{1}{4}z_{31}Y_2 - \frac{1}{4}z_{13}z_{31}Y_3
\\
\hphantom{[\zeta_2, \zeta_7] = }{}
+ \biggl(\frac{1}{4}z_{11} - \frac{1}{4}z_{14}z_{31}\biggr)Y_4 =: \zeta_9
\\
[\zeta_3, \zeta_5] = - \frac{\pa}{\pa z_{11}} - \frac{1}{2}z_{31}\frac{\pa}{\pa z_{15}}
+ \biggl(\frac{1}{2}z_{21} - \frac{1}{2}z_{24}z_{31}\biggr)\frac{\pa}{\pa z_{16}} = - [\zeta_1, \zeta_6] = - \zeta_8,
\\
[\zeta_3, \zeta_6] = -\frac{1}{2}\frac{\pa}{\pa z_{25}} =: \zeta_{10}, \qquad [\zeta_3, \zeta_7] = 0,
\\
[\zeta_4, \zeta_5] = 0,
\qquad
[\zeta_4, \zeta_6] = - [\zeta_2, \zeta_7] = - \zeta_9, \qquad [\zeta_4, \zeta_7] = 0.
\end{gather*}
Then we have \smash{$\rank E^{(3)} = 10$}.

Further we have
\begin{gather*}
[\zeta_1, \zeta_8] = 0,\\
[\zeta_1, \zeta_9] = \biggl(-\frac{1}{2}z_{25} + \frac{1}{8}z_{21}z_{31}\biggr)X_1 + \frac{1}{2}X_2 - \frac{1}{4}z_{31}Y_3 + \biggl(\frac{1}{4}z_{21} - \frac{1}{4}z_{24}z_{31}\biggr)Y_4 =: \zeta_{11},
\\
[\zeta_1, \zeta_{10}] = \frac{1}{2}\frac{\pa}{\pa z_{15}} + \frac{1}{2}z_{24}\frac{\pa}{\pa z_{16}} =: \zeta_{12},
\\
[\zeta_2, \zeta_8] = 0, \qquad [\zeta_2, \zeta_9] = 0, \qquad [\zeta_2, \zeta_{10}] = 0,
\\
[\zeta_3, \zeta_8] = [\zeta_1, \zeta_{10}] = \zeta_{12}, \qquad
[\zeta_3, \zeta_9] = \frac{1}{8}z_{11}X_1 - \frac{1}{4}Y_2 - \frac{1}{4}z_{13}Y_3 - \frac{1}{4}z_{14}Y_4 =: \zeta_{13}, \\ [\zeta_3, \zeta_{10}] = 0,
\\
[\zeta_4, \zeta_8] = - [\zeta_1, \zeta_9] = - \zeta_{11}, \qquad [\zeta_4, \zeta_9] =0, \qquad
[\zeta_4, \zeta_{10}] = -2[\zeta_3, \zeta_9] = -2\zeta_{13}.
\end{gather*}
We get that \smash{$\rank E^{(4)} = 13$}.

Further we have
\begin{gather*}
[\zeta_1, \zeta_{11}] = 0, \qquad [\zeta_1, \zeta_{12}] = 0, \qquad
[\zeta_1, \zeta_{13}] = \frac{1}{8}z_{21}X_1 - \frac{1}{4}Y_3 - \frac{1}{4}z_{24}Y_4 =: \zeta_{14},\\
[\zeta_2, \zeta_{11}] = 0,\qquad
[\zeta_2, \zeta_{12}] = \frac{1}{2}\frac{\pa}{\pa z_{16}} =: \zeta_{15}, \qquad [\zeta_2, \zeta_{13}] = 0, \\
[\zeta_3, \zeta_{11}] = [\zeta_1, \zeta_{13}] = \zeta_{14}, \qquad [\zeta_3, \zeta_{12}] = 0,\qquad
[\zeta_3, \zeta_{13}] = 0, \\
[\zeta_4, \zeta_{11}] = 0, \qquad
[\zeta_4, \zeta_{12}] = - \frac{1}{4}z_{21}X_1 + \frac{1}{2}Y_3 + \frac{1}{2}z_{24}Y_4 = -2[\zeta_1, \zeta_{13}] =
-2\zeta_{14},\\
[\zeta_4, \zeta_{13}] = \biggl(-\frac{1}{8}z_{11}^2 - \frac{1}{2}z_{13}z_{16} + \frac{1}{2}z_{14}z_{15}\biggr)X_{12}
+ \frac{1}{2}z_{16}X_{13} - \frac{1}{2}z_{15}X_{14} - \frac{1}{2}z_{14}X_{23}\\
\hphantom{[\zeta_4, \zeta_{13}] =}{}
 + \frac{1}{2}z_{13}X_{24} - \frac{1}{2}X_{34} + \frac{1}{4}z_{11}Z =: \zeta_{16}.
\end{gather*}
Therefore, we have \smash{$\rank E^{(5)} = 16$}.

Furthermore,
\begin{gather*}
[\zeta_1, \zeta_{14}] = 0, \qquad [\zeta_1, \zeta_{15}] = 0, \qquad [\zeta_1, \zeta_{16}] = 0, \\
[\zeta_2, \zeta_{14}] = \frac{1}{8}z_{31}X_1 - \frac{1}{4}Y_4 =: \zeta_{17}, \qquad
[\zeta_2, \zeta_{15}] = 0, \qquad [\zeta_2, \zeta_{16}] = 0, \\
[\zeta_3, \zeta_{14}] = 0, \qquad [\zeta_3, \zeta_{15}] = 0, \qquad [\zeta_3, \zeta_{16}] = 0,\\
[\zeta_4, \zeta_{14}] = \biggl(-\frac{1}{2}z_{16} - \frac{1}{4}z_{11}z_{21} + \frac{1}{2}z_{14}z_{25}
+ \frac{1}{2}z_{15}z_{24} + \frac{1}{8}z_{13}z_{21}^2 - \frac{1}{2}z_{13}z_{24}z_{25}\biggr)X_2\\
\hphantom{[\zeta_4, \zeta_{14}] =}{}
+\biggl(-\frac{1}{8}z_{21}^2 + \frac{1}{2}z_{24}z_{25}\biggr)X_{13} - \frac{1}{2}z_{25}X_{14} - \frac{1}{2}z_{24}X_{23} + \frac{1}{2}X_{24} + \frac{1}{4}z_{12}Z =: \zeta_{18},\\
[\zeta_4, \zeta_{15}] = -2[\zeta_2, \zeta_{14}] = -2\zeta_{17}, \qquad
[\zeta_4, \zeta_{16}] = 0.
\end{gather*}
Thus we see \smash{$\rank E^{(6)} = 18$}.

Furthermore, we have
\begin{gather*}
[\zeta_1, \zeta_{17}] = 0, \qquad [\zeta_1, \zeta_{18}] = 0, \qquad [\zeta_2, \zeta_{17}] = 0, \\
[\zeta_2, \zeta_{18}] = \biggl(\frac{1}{2}z_{15} - \frac{1}{4}z_{11}z_{31} - \frac{1}{2}z_{13}z_{25} + \frac{1}{8}z_{14}z_{31}^2
+ \frac{1}{4}z_{13}z_{21}z_{31} - \frac{1}{8}z_{13}z_{24}z_{31}^2\biggr)X_{12}
\\
\hphantom{[\zeta_2, \zeta_{18}] =}{}
+ \biggl(\frac{1}{2}z_{25} - \frac{1}{4}z_{21}z_{31} + \frac{1}{8}z_{24}z_{31}^2\biggr)X_{13} - \frac{1}{8}z_{31}^2X_{14} - \frac{1}{2}X_{23} + \frac{1}{4}z_{31}Z =: \zeta_{19},
\\
[\zeta_3, \zeta_{17}] = \frac{1}{8}X_1 =: \zeta_{20}, \qquad [\zeta_3, \zeta_{18}] = 0, \\
[\zeta_4, \zeta_{17}] = [\zeta_2, \zeta_{18}] = \zeta_{19}, \qquad [\zeta_4, \zeta_{18}] = 0.
\end{gather*}
Thus we have
\smash{$\rank E^{(7)} = 20$}.

Moreover,
\begin{gather*}
[\zeta_1, \zeta_{19}] = 0, \qquad [\zeta_1, \zeta_{20}] = 0, \\
[\zeta_2, \zeta_{19}] = 0, \qquad [\zeta_2, \zeta_{20}] = 0, \\
[\zeta_3, \zeta_{19}] = \biggl(-\frac{1}{4}z_{11} + \frac{1}{4}z_{14}z_{31} + \frac{1}{4}z_{13}z_{21} - \frac{1}{4}z_{13}z_{24}z_{31}\biggr)X_{12} +
\biggl(-\frac{1}{4}z_{21} + \frac{1}{4}z_{24}z_{31}\biggr)X_{13}
\\
\hphantom{[\zeta_3, \zeta_{19}] =}{}
 - \frac{1}{4}z_{31}X_{14} + \frac{1}{4}Z =: \zeta_{21},
\\
[\zeta_3, \zeta_{20}] = 0, \\
[\zeta_4, \zeta_{19}] = 0, \qquad [\zeta_4, \zeta_{20}] = \frac{1}{2}[\zeta_3, \zeta_{19}]
= \frac{1}{2}\zeta_{21}.
\end{gather*}
We obtain that \smash{$\rank E^{(8)} = 21$}.

We have
\begin{gather*}
[\zeta_1, \zeta_{21}] = 0, \qquad
[\zeta_2, \zeta_{21}] = 0, \\
[\zeta_3, \zeta_{21}] =
\biggl(\frac{1}{4}z_{14} - \frac{1}{4}z_{13}z_{24}\biggr)X_{12} + \frac{1}{4}z_{24}X_{13} - \frac{1}{4}X_{14} =: \zeta_{22}
\end{gather*}
and we have $[\zeta_4, \zeta_{21}] = 0$.
So we get that \smash{$\rank E^{(9)} = 22$}.

Also we have
\begin{gather*}
[\zeta_1, \zeta_{22}] = 0, \qquad [\zeta_2, \zeta_{22}] = -\frac{1}{4}z_{13}X_{12} + \frac{1}{4}X_{13} =: \zeta_{23}, \\
 [\zeta_3, \zeta_{22}] = 0, \qquad [\zeta_4, \zeta_{22}] = 0.
\end{gather*}
Then we have \smash{$\rank E^{(10)} = 23$}.

Lastly, we have
\begin{alignat*}{3}
&[\zeta_1, \zeta_{23}] = - \frac{1}{4}X_{12} =: \zeta_{24}, \qquad &&[\zeta_2, \zeta_{23}] = 0,& \\
&[\zeta_3, \zeta_{23}] = 0, \qquad&& [\zeta_4, \zeta_{23}] = 0.&
\end{alignat*}
We have that $\rank E^{(11)} = 24$.
This shows the claim.
\end{proof}

\begin{Remark}
By the calculations in the proof of Lemma~\ref{F_4-bracket-relations}, we observe that
$\pi_*^{-1}(D) \subset E^{(7)}$ for the projection $\pi\colon W \to M$, $\pi(m, (V_1, V_2, V_4)) = m$.
\end{Remark}

\section[(8,15)-distributions of type F\_4]{$\boldsymbol{(8,15)}$-distributions of type $\boldsymbol{F_4}$}
\label{$(8, 15)$-distributions of type $F_4$}

Inspired by our study on singular curves for Cartan model performed in the previous sections,
it would be natural to introduce the class of $(8, 15)$-distributions of type $F_4$ including Cartan's model.

\begin{Definition}
\label{(8, 15)-distribution-of-type-F_4}
Let $D \subset TM$ be a complex (resp.\ a real) $(8, 15)$-distribution.
Then we call $D$ {\it of type $F_4$} (resp.\ {\it of type $F_{4(4)}$})
if, for each point $x_0 \in M$, there exists a local frame $\{ X_1, X_2, X_3, X_4,
Y_1, Y_2, Y_3, Y_4\}$ of $D$ over an open neighbourhood of $x_0$ such that, modulo ${\mathcal D}$,
\begin{gather*}
[X_1, X_2] \equiv [Y_3, Y_4], \qquad [X_1, X_3] \equiv -[Y_2, Y_4], \qquad [X_1, X_4] \equiv [Y_2, Y_3],
\\
[X_2, X_3] \equiv [Y_1, Y_4], \qquad [X_2, X_4] \equiv -[Y_1, Y_3], \qquad [X_3, X_4] \equiv [Y_1, Y_2],
\\
[X_1, Y_1] \equiv [X_2, Y_2] \equiv [X_3, Y_3] \equiv [X_4, Y_4],
\qquad {\text{and}} \qquad [X_i, Y_j] \equiv 0 \quad i \not= j, \ 1 \leq i, j \leq 4,
\end{gather*}
and, if we set
\begin{alignat*}{4}
&X_{12} = \frac{1}{2}[X_1, X_2],\qquad&& X_{13} = \frac{1}{2}[X_1, X_3],\qquad&& X_{14} = \frac{1}{2}[X_1, X_4],&\\
&X_{23} = \frac{1}{2}[X_2, X_3],\qquad&& X_{24} = \frac{1}{2}[X_2, X_4],\qquad&& X_{34} = \frac{1}{2}[X_3, X_4],&
\end{alignat*}
and $Z = [Y_1, X_1]$,
then the vector fields
\[
X_1,\ X_2,\ X_3,\ X_4,\ Y_1,\ Y_2,\ Y_3,\ Y_4\, X_{12},\ X_{13},\ X_{14},\ X_{23},\ X_{24},\ X_{34},\ Z,
\]
form a local frame of $TM$.
\end{Definition}

\begin{Remark}
Comparing with the relations on generators of Cartan's model in Section~\ref{Cartan-model},
the~relations in Definition~\ref{(8, 15)-distribution-of-type-F_4} are given modulo ${\mathcal D}$.
The class of $(8, 15)$-distributions of type $F_4$ in~Definition~\ref{(8, 15)-distribution-of-type-F_4}
coincides with the class of regular differential system of type ${\mathfrak m}_F$ in the sense of
Tanaka \cite{Tanaka1, Tanaka2, Yamaguchi}.
\end{Remark}

Then we have the following theorem.

\begin{Theorem}\label{Main-Theorem}
 Let $(M, D)$ be a complex $($resp.\ real$)$ $(8, 15)$-distribution of type $F_4$ $($resp.\ $F_{4(4)})$.
Then there exist uniquely the conformal non-degenerate bilinear form $($resp.\ $(4, 4)$-metric$)$ on~$D$ and the conformal non-degenerate bilinear form $($resp.\ $(4, 3)$-metric$)$ on $D^\perp$ obtained from the abnormal bi-extremals of $D$
such that the null-cone $C \subset D$ coincides with the singular velocity cone~$\SVC(D)$.
Moreover, the flag manifold of null-subspaces
${\bigl\{ \Lambda_1 \subset \Lambda_2 \subset \Lambda_3 \subset D^\perp \subset T^*M\bigr\}}$
corresponds to a subclass of flags by null-subspaces ${\{ V_1 \subset V_2 \subset V_4 \subset C \subset D \subset TM\}}$ in $D$. The prolongation $(W, E)$ of $(M, D)$ by the above null-flags of $D$
turns out to be a~$(4,7,10,13,16,18,20,\allowbreak 21,22,23,24)$-distribution such that its symbol algebra
is isomorphic to the negative part of the nilpotent algebra for the gradation by the full set
$\{\alpha_1, \alpha_2, \alpha_3, \alpha_4\}$ of simple roots
of simple Lie algebra~$F_4$ $($resp.\ $F_{4(4)})$.
\end{Theorem}

\begin{proof}[Proof of Theorem~\ref{Main-Theorem}]
We re-examine the arguments on Cartan's model of $(8, 15)$-dis\-tri\-bu\-tion
defined in Section~\ref{Cartan-model} and performed in Sections~\ref{Singular-velocity}--\ref{Prolongation-of-Cartan's-$(8, 15)$-distribution} for general $(8, 15)$-dis\-tri\-bu\-tions of type~$F_4$.

Let $D \subset TM$ be an $(8, 15)$-distribution of type $F_4$.
Reversing the correspondence in Section~\ref{Cartan-model},
we take the local frame
\[
\beta_1,\ \beta_2,\ \beta_3,\ \beta_4,\ \gamma_1,\ \gamma_2,\ \gamma_3,\ \gamma_4,\ \omega_{12},\ \omega_{13},\
\omega_{14},\ \omega_{23},\ \omega_{24},\ \omega_{34},\ \sigma
\]
of $T^*M$ which is dual to the local frame
\[
X_1,\ X_2,\ X_3,\ X_4,\ Y_1,\ Y_2,\ Y_3,\ Y_4,\ X_{12},\ X_{13},\ X_{14},\ X_{23},\ X_{24},\ X_{34},\ Z
\]
of $TM$ in Definition~\ref{(8, 15)-distribution-of-type-F_4}. Then $D^\perp$ is
generated by $\omega_{12}$,
$\omega_{13}$, $\omega_{14}$, $\omega_{23}$, $\omega_{24}$, $\omega_{34}$ and $\sigma$.
Any~$\alpha \in D^\perp$ is expressed uniquely as
\smash{$\alpha = \sum_{1 \leq i < j \leq 4} r_{ij} \omega_{ij} + s\sigma$}. Then
we have $\langle \alpha, X_{ij}\rangle = r_{ij}$ and $\langle \alpha, \sigma\rangle = s$.
The functions $r_{ij}$ and $s$ with local coordinates of the base manifold
$M$ form a system of local coordinates of the submanifold $D^\perp \subset T^*M$.
Then the equations \eqref{star} and \eqref{star-star} are obtained
other (linear) algebraic arguments in Section~\ref{Singular-velocity} work as well
also for general $(8, 15)$-distributions of type $F_4$. Thus we have
the same conclusion of Corollary~\ref{Conformal-metrics} and moreover our discussions
on the correspondence of null-flags in $D$ and $D^\perp$
performed in Section~\ref{Null-flags} and the same proofs
of the results such as Lemma~\ref{F_4-bracket-relations}
which concern on the prolongations of $D$ in Section~\ref{Prolongation-of-Cartan's-$(8, 15)$-distribution}
works well also for any $(8, 15)$-distribution of type $F_4$. This shows Theorem~\ref{Main-Theorem}.
\end{proof}

\begin{Remark}
The above statement on $(8, 15)$-distribution of type $F_4$ (resp.\ $F_{4(4)}$) means that
the gradation sheaf, i.e.,
the sheaf of nilpotent graded Lie algebras
\smash{${\mathfrak m} := \bigoplus_{i=1}^{11}\bigl({\mathcal D}^{(i)}/{\mathcal D}^{(i-1)}\bigr)$}
is isomorphic to that for the model derived from
the simple Lie algebra $F_4$, which is described in Section~\ref{Cartan-model}.
It is stated in \cite{Yamaguchi} (see Proposition 5.5 and the arguments in pp.~482--483)
that any $(8, 15)$-distribution of type $F_4$ (resp.\ \smash{$F_{4(4)}$}) is isomorphic to Cartan's model over $\C$ (resp.\ $\R$)
in fact by Tanaka theory on simple graded Lie algebras.
Note that we have proved our Theorem~\ref{Main-Theorem} without using this fact.
\end{Remark}

\subsection*{Acknowledgements}
The authors would like to thank anonymous referees for valuable and helpful comments to improve the paper. The first author is partially supported by JSPS KAKENHI Grant Number 24K06700,
by JST CREST Geometrical Understanding of Spatial Orientation and by
the Research Institute for Mathematical Sciences in Kyoto University.

\pdfbookmark[1]{References}{ref}
\LastPageEnding

\end{document}